\documentclass[12pt,leqno,a4paper]{amsart}
\usepackage{amssymb,enumerate}
\newcommand{\CC}{{\mathbb{C}}}
\newcommand{\FF}{{\mathbb{F}}}
\newcommand{\RR}{{\mathbb{R}}}
\newcommand{\ZZ}{{\mathbb{Z}}}

\newcommand{\fA}{{\mathfrak{A}}}
\newcommand{\fS}{{\mathfrak{S}}}

\newcommand{\bC}{{\mathbf{C}}}
\newcommand{\bG}{{\mathbf{G}}}
\newcommand{\bH}{{\mathbf{H}}}
\newcommand{\bL}{{\mathbf{L}}}
\newcommand{\bT}{{\mathbf{T}}}
\newcommand{\bS}{{\mathbf{S}}}
\newcommand{\bU}{{\mathbf{U}}}
\newcommand{\bP}{{\mathbf{P}}}
\newcommand{\bQ}{{\mathbf{Q}}}
\newcommand{\bX}{{\mathbf{X}}}
\newcommand{\bB}{{\mathbf{B}}}

\newcommand{\cE}{{\mathcal{E}}}

\newcommand{\End}{{\operatorname{End}}}
\newcommand{\Aut}{{\operatorname{Aut}}}
\newcommand{\Inndiag}{{\operatorname{Inndiag}}}
\newcommand{\Irr}{{\operatorname{Irr}}}
\newcommand{\St}{{\operatorname{St}}}
\newcommand{\tr}{{\operatorname{tr}}}
\newcommand{\rk}{{\operatorname{rank}}}
\newcommand{\Ind}{{\operatorname{Ind}}}
\newcommand{\Char}{{\operatorname{char}\,}}
\newcommand{\diag}{{\operatorname{diag}\,}}
\newcommand{\GL}{{\operatorname{GL}}}
\newcommand{\SL}{{\operatorname{SL}}}
\newcommand{\PGL}{{\operatorname{PGL}}}
\newcommand{\PSL}{{\operatorname{L}}}
\newcommand{\PSU}{{\operatorname{U}}}
\newcommand{\GU}{{\operatorname{GU}}}
\newcommand{\SU}{{\operatorname{SU}}}
\newcommand{\Sp}{{\operatorname{Sp}}}
\newcommand{\CSp}{{\operatorname{CSp}}}
\newcommand{\PSp}{{\operatorname{S}}}
\newcommand{\OO}{{\operatorname{O}}}

\newcommand{\CSpin}{{\operatorname{CSpin}}}
\newcommand{\SO}{{\operatorname{SO}}}

\newcommand{\Chevie}{{\sf Chevie}}
\newcommand{\al}{\alpha}

\def\pmod#1{~({\rm mod}~#1)}
\newcommand{\tw}[1]{{}^#1\!}

\let\eps=\epsilon
\newtheorem{thm}{Theorem}[section]
\newtheorem{lem}[thm]{Lemma}
\newtheorem{cor}[thm]{Corollary}
\newtheorem{prop}[thm]{Proposition}
\newtheorem*{conj}{Conjecture A}

\theoremstyle{definition}

\newtheorem{exmp}[thm]{Example}

\theoremstyle{remark}
\newtheorem{rem}[thm]{Remark}

\begin{document}

\title[Products of Conjugacy Classes]
{Products of Conjugacy Classes in Finite and Algebraic Simple Groups}

\date{\today}

\author{Robert Guralnick}
\address{Department of Mathematics, University of
  Southern California, Los Angeles, CA 90089-2532, USA}
\makeatletter\email{guralnic@usc.edu}\makeatother
\author{Gunter Malle}
\address{FB Mathematik, TU Kaiserslautern, Postfach 3049,
  67653 Kaisers\-lautern, Germany}
\makeatletter\email{malle@mathematik.uni-kl.de}\makeatother
\author{Pham Huu Tiep}
\address{Department of Mathematics, University of Arizona, Tucson,
  AZ 85721-0089, USA}
\makeatletter\email{tiep@math.arizona.edu}\makeatother

\thanks{The first and third authors were partially supported by the NSF
  (grants DMS-1001962 and DMS-0901241).}



\begin{abstract}
We prove the Arad--Herzog conjecture for various families of finite simple
groups --- if $A$ and $B$ are nontrivial conjugacy classes,
then $AB$ is not a conjugacy class. We also prove that if $G$ is a finite
simple group of Lie type and $A$ and $B$ are nontrivial conjugacy classes,
either both semisimple or both unipotent, then $AB$ is not a conjugacy class.
We also prove a strong version of the Arad--Herzog conjecture for simple
algebraic groups and in particular show that almost always the product of
two conjugacy classes in a simple algebraic group consists of infinitely
many conjugacy classes.   As a consequence we obtain a complete classification
of pairs of centralizers in a simple algebraic group which have dense product.
A special case of this has been used by Prasad to prove a uniqueness result for
Tits systems in pseudo-reductive groups. 
Our final result is a generalization of the Baer-Suzuki theorem for 
$p$-elements with $p \geq 5$. 
\end{abstract}

\dedicatory{Dedicated to Robert Steinberg on the occasion of his 90th birthday}

\maketitle

%%\pagestyle{myheadings}
%\markboth{for personal use only}{preliminary version}
%%\markboth{}{}

%%%%%%%%%%%%%%%%%%%%%%%%%%%%%%%%%%%%%%%%%%%%%%%%%%%%%%%%%%%%%%%%%%%%%%%%%
\section{Introduction} \label{sec:intro}

In \cite[p.~3]{AH}, Arad and Herzog made the following conjecture:

\begin{conj}[Arad--Herzog]   \label{conj:AH}
 If $S$ is a finite non-abelian simple group and $A$ and $B$ are nontrivial
 conjugacy classes of $S$, then $AB$ is not a conjugacy class.
\end{conj}

In this paper, we prove this conjecture in various cases. We
also consider the analogous problem for simple algebraic groups. Note
that the results do not depend on the isogeny class of the group
(allowing the possibility of multiplying a class by a central element)
and so we work with whatever form is more convenient.  Moreover, in 
characteristic $2$, we ignore the groups of type $B$ (the result can
be read off from the groups of type $C$).

Here one can prove much more:

\begin{thm}  \label{main:alg}
 Let $\bG$ be a simple algebraic group over an algebraically closed field of
 characteristic $p \ge 0$. Let $A$ and $B$ be non-central conjugacy classes
 of $\bG$. Then $AB$ can never constitute a single conjugacy class. In fact, 
 either $AB$ is the union of infinitely many conjugacy classes, or
 (up to interchanging $A$ and $B$ and up to an isogeny for $\bG$) 
  one of the following holds:
 \begin{enumerate}[\rm(1)]
  \item $\bG=G_2$, $A$ consists of long root elements and $B$ consists of
   elements of order $3$. If $p=3$, $B$ consists of short root elements and
   if $p \ne 3$, $B$ consists of elements with centralizer isomorphic
   to $\SL_3$.
  \item $\bG=F_4$, $A$ consists of long root elements and $B$ consists of
   involutions. If $p=2$, $B$ consists of short root elements and if $p \ne 2$,
   $B$ consists of involutions with centralizer isomorphic to $B_4$.
  \item $\bG=\Sp_{2n}= \Sp(V)$, $n\ge 2$, $\pm A$ consists of long root
   elements and $B$ consists of involutions; when $p=2$ then the involutions
   $b\in B$ moreover satisfy $(bv,v)=0$ for all $v \in V$.
  \item $\bG=\SO_{2n+1}$, $n \ge 2$, $p \ne 2$, $A$ consists of elements
   which are the negative of a reflection and $B$ consists of unipotent
   elements with all Jordan blocks of size at most~$2$.
 \end{enumerate}
\end{thm}

The methods rely heavily on closure of unipotent classes. In particular, this
gives a short proof for simple algebraic groups of what is referred to as
Szep's conjecture for the finite simple groups (proved in \cite{FA}) ---
a finite simple group is not the product of two subgroups with nontrivial
centers.

\begin{cor}   \label{alg:factorization}
 Let $\bG$ be a simple algebraic group over an algebraically closed field
 of characteristic $p \ge 0$. Let $a, b$ be non-central elements of $\bG$.
 Then $\bG \ne C_{\bG}(a)C_{\bG}(b)$.
\end{cor}

Indeed, we see that $C_{\bG}(a)C_{\bG}(b)$ is rarely dense in 
$\bG$ (it only happens in
the exceptional cases in Theorem~\ref{main:alg}) --- see
Corollary~\ref{cor:factorize}. In particular, we give a very short proof of:

\begin{cor}   \label{cor:prasad}
 If $\bG$ is a simple algebraic group and $x$ is a non-central element of $G$,
 then for any $g \in G$,  $C_{\bG}(x)gC_{\bG}(x)$ is not dense in $G$.
 In particular,  $|C_{\bG}(x) \backslash \bG / C_{\bG}(x)|$ is infinite.
\end{cor}

This  was proved independently for unipotent
elements by Liebeck and Seitz \cite[Chapter~1]{LSbook}.  The previous result
was used by Prasad \cite[Thm.~B]{prasad} to show that any
Tits system for a pseudo-reductive group satisfying some natural conditions
is a standard Tits system (see \cite{prasad} for more details).

Conjecture A  is open only for the simple groups of Lie type,
where it was known to be true for certain families (cf.~\cite{MoTo}),
but not for any family of arbitrary rank and field size.
Our idea is to show that we can find a small set of irreducible
characters $\Gamma$ of $S$ so that for any pair of nontrivial classes
$A,B\subset S$ there is $\chi\in\Gamma$ which is not constant on $AB$ (and
so obviously $AB$ is not a conjugacy class).

For $\fA_n$ and $\fS_n$, the conjecture was proved by Fisman and Arad
\cite[3.1]{FA}; see also Adan-Bante and Verrill \cite{AV}. In
Section~\ref{sec:sym} we give a very short proof of the slightly stronger
result:

\begin{thm}   \label{thm:alt}
 Let $H :=\fA_n$ and $G :=\fS_n$ with $n \geq 5$. For nontrivial elements
 $a,b \in G$, set $A :=a^H$ and $B :=b^H$. For $g \in G$,
 let $f(g)$ denote the number of fixed points of $g$ in the natural
 permutation representation of $G$. Then $f$ is not constant on $AB$.
\end{thm}

Similarly, we show:

\begin{thm}   \label{thm:main1}
 Let $S = \PSL_d(q)= \PSL(V)$ be simple. Let $f(g)$ be the number of fixed
 one-spaces of $g\in S$ on $V$. If $A$ and $B$ are nontrivial conjugacy
 classes of $S$, then $f$ is not constant on $AB$ (in particular, $AB$ is
 not a conjugacy class).
\end{thm}

As noted above 
this is the first family of groups of Lie type including both unbounded rank
and field size for which the conjecture is now established.

For arbitrary groups of Lie type using Lemma \ref{lem:basic}, the fact that
the Steinberg character is nonzero on semisimple elements only and the result
that the product of centralizers of two non-central semisimple elements in a
simple algebraic group is not dense (cf. Corollary \ref{cor:factorize}),
we can show:

\begin{thm}  \label{thm:main2}
 Let $G$ be a finite simple group of Lie type, and let $\St$ denote the
 Steinberg character of $G$. If $a, b \in G\setminus\{1\}$ are semisimple
 elements, then $\St$ is not constant on~$a^Gb^G$.
\end{thm}

This implies immediately:

\begin{cor}  \label{cor:ss}
 Let $G$ be a finite simple group of Lie type and $a,b,c\in G\setminus\{1\}$
 such that $a^Gb^G=c^G$. Then neither $c$ is semisimple, nor are both $a,b$.
\end{cor}

Using Deligne--Lusztig theory, we can similarly show:

\begin{thm} \label{thm:main3}
 Let $G$ be a finite simple group of Lie type and $a,b,c\in G\setminus\{1\}$
 such that $a^Gb^G=c^G$. Then $c$ is not unipotent and so neither
 are both $a$ and $b$.
\end{thm}

Recall that the Baer--Suzuki theorem states that if $G$ is a finite group,
$p$ a prime and $x\in G$ is such that $\langle x, x^g \rangle$ is a $p$-group
for all $g \in G$, then the normal closure of $x$ in $G$ is a $p$-group.
One step in the proof of Theorem \ref{thm:main3} is an analog of
Theorem~\ref{main:alg} for pairs of $p$-elements of finite groups.
This leads to a generalization of the Baer--Suzuki theorem (for primes at least
$5$) by considering two possibly distinct conjugacy classes of $p$-elements.

\begin{thm}   \label{thm:bsthm}
 Let $G$ be a finite group, $p \ge 5$ prime, and let $C$ and $D$ be normal
 subsets of $G$ with $H:=\langle C \rangle = \langle D \rangle$. Suppose that
 for every pair $(c,d) \in C \times D$, $\langle c, d \rangle$ is a
 $p$-group. Then $H$ is a $p$-group.
\end{thm}

The Baer--Suzuki theorem (for $p$-elements with $p \ge 5$) is the special
case $C=D$. Example~\ref{ex:wreath} shows that we cannot drop
the assumption that $\langle C \rangle = \langle D \rangle$.
The examples in Section \ref{sec:dense} show that we must also require
that $p \ge 5$.

See Theorems~\ref{thm:bsas} and~\ref{thm:bsgen} for other variants.

\medskip
This paper is organized as follows. In Section~\ref{sec:sym} we write down
a variant of the character-theoretic condition for a product of
conjugacy classes to be a conjugacy class in a finite group, and then use it to
give short proofs of Conjecture~A  for $\fA_n$ and $\PSL_d(q)$.
In Sections~\ref{sec:classic} and~\ref{sec:exc}, we show that the conjecture
holds for low rank classical and exceptional groups, and prove
Theorems ~\ref{thm:main2} and 
%develop the machinery needed to prove Theorem 
\ref{thm:main3}.

In Section \ref{sec:alg}, we consider algebraic groups and prove
Theorem~\ref{main:alg}. We also establish Corollary~\ref{alg:factorization}, and
classify in Corollary~\ref{cor:factorize} the cases when products of
centralizers in a simple algebraic group over an algebraically closed field
are dense. We then discuss in Section \ref{sec:dense} the special cases listed in
Theorem~\ref{main:alg} in detail. These two sections are essentially 
independent of the rest of the paper (only Corollary \ref{cor:factorize} is 
used to prove Theorem~\ref{thm:main2} for the finite groups of Lie type).

In the next section, we use our results on semisimple elements to give
a relatively quick proof of Szep's conjecture and also provide some examples
which show that the simplicity hypothesis in both Conjecture~A 
and Szep's conjecture cannot be weakened much.

In the last section, we prove variants of Theorem~\ref{thm:bsthm}.

\medskip
In order to prove Conjecture~A 
for the remaining open cases, one will have to work much harder.
The short proofs for the alternating groups and linear groups used the fact
that the groups had doubly transitive permutation representations
(however, the proof does not work for all doubly transitive simple groups).
There are a few other special cases where the existence of a special character
makes the proofs relatively straightforward.  The conjecture can be checked
easily for the finite groups of Lie type of small rank using \Chevie.
In a sequel, employing more sophisticated
tools from the representation theory of finite groups of Lie type, we hope to
establish Conjecture A. We have proved the result for several
families of classical groups -- in particular the conjecture holds for
symplectic groups (at this point the proof of this case is roughly $40$ pages
long). The methods here depend upon proving some new results about character
values for these groups.

\begin{rem}
A dual problem to considering products of conjugacy classes would be
to consider tensor products of irreducible representations.   
See \cite{BeK1, BeK2, MMT, MT01, Z} for some partial results. 
\end{rem}

\vskip 1pc
\noindent{\bf Acknowledgements:}
It is a pleasure to thank Ross Lawther for writing his interesting paper
\cite{law} at our request and also for allowing us to include his result,
Lemma \ref{f4 lemma}. We also thank Tim Burness and Gopal Prasad 
for some helpful comments.

%%%%%%%%%%%%%%%%%%%%%%%%%%%%%%%%%%%%%%%%%%%%%%%%%%%%%%%%%%%%%%%%%%%%%%%%%
\section{$\fS_n$, $\fA_n$, and Projective Linear Groups}  \label{sec:sym}
We start by proving Theorem \ref{thm:alt} which we restate below:

\begin{thm}
 Let $a, b\in\fS_n\setminus\{1\}$ with $n \geq 5$ and set $A :=a^{\fA_n}$ and
 $B :=b^{\fA_n}$. For $g \in \fS_n$, let $f(g)$ be the number of fixed points
 of $g$ in the natural permutation representation. Then $f$ is not constant
 on $AB$.
\end{thm}

The proof uses the following easy lemma.

\begin{lem}  \label{lem:basic}
 Let $G$ be a finite group with $H$ a subgroup of $G$. Let $a,b\in G$ and
 set $c=ab$, $A=a^H$ and $B=b^H$. Let $V$ be an irreducible $\CC G$-module
 that remains irreducible for $H$. If $\chi$ is the character of $V$ and
 $\chi$ is constant on $AB$, then $\chi(a)\chi(b)=\chi(c)\chi(1)$.
\end{lem}

\begin{proof}
For $X \subseteq G$, let $\theta(X)= \sum_{x \in X} x \in \ZZ G$.
Write $\theta(A)\theta(B) = \sum e_i \theta(C_i)$ where $C_i$ are the
$H$-orbits of elements in $AB$. Let $\rho~:~G \rightarrow \GL(V)$ denote
the representation of $G$ on $V$. Since $H$ acts irreducibly on $V$, it follows
that if $D = d^H $ for some $d \in G$, then $\rho(\theta(D))$ acts as a
scalar on $V$. Computing traces, we see that the scalar is given by
$$ \frac{ |D|\,\chi(d)}{\chi(1)}. $$

Thus,
$$\frac{|A| |B| \chi(a)\chi(b)}{\chi(1)^2}
  = (\sum_i e_i|C_i|)\frac{\chi(c)}{\chi(1)}.
$$
Since $|A||B|= \sum e_i |C_i|$, the result follows.
\end{proof}

\begin{proof}[Proof of Theorem~\ref{thm:alt}]
For $n = 5$, one checks directly. So assume the theorem is false
for some $n > 5$. Let $a \in A$, $b \in B$ and $c \in AB$.
Note that $\chi:=f-1$ is an irreducible character of both $\fS_n$ and $\fA_n$.
If $a$ and $b$ each have a fixed point, then the result follows by induction.

So we may assume that $f(a)=0$, i.e., $\chi(a)= -1$. Since
$\chi(a)\chi(b)=(n-1)\chi(c)$ by Lemma~\ref{lem:basic}, $\chi(c) \ne 0$
implies that $b=1$, a contradiction. So $\chi(b)=0 = \chi(c)$.
In particular, $c$ has a unique fixed point.

Suppose that neither $a$ nor $b$ is an involution.
Then $a$ and $b$ both contain cycles of length at least $r \ge 3$.
We can then replace $b$ by a conjugate so that $ab$ has at least
$r-1 \ge 2$ fixed points, a contradiction.

Suppose that either $b$ has a cycle of length $4$ or at least $2$ nontrivial
cycles. Thus, arguing as above, if $a$ is not a $2$-cycle (in the first case)
or a $2$-cycle or $3$-cycle (in the second case), we can arrange
for $ab$ to have at least $2$ fixed points, a contradiction.

If $a$ is a $2$-cycle, we can reduce to the case that $b$ is an
$m-1$-cycle on $m$ points. Then $ab$ can be an $m$-cycle or can
have fixed points, a contradiction. Similarly if $a$ is a $3$-cycle,
we can reduce to the case $b$ is an $m-1$-cycle on $m$ points.
Again, we can arrange for $ab$ either to have fixed points or not,
a contradiction.
\end{proof}

We next consider $\PSL_d(q)$. We first note a much stronger
result for $d=2$.

\begin{lem}   \label{lem:d=2}
 Let $a, b \in \GL_2(q)$, $q>3$ with $a,b$ both non-central.
 Set $A=a^H$ and $B=b^H$ where $H = \SL_2(q)$.
 \begin{enumerate}[\rm(a)]
  \item There exist $(u_i,v_i) \in A \times B$, $i=1,2$ such that $u_1v_1$
   fixes a lines and $u_2v_2$ does not.
  \item If $a$ and $b$ are semisimple elements with an eigenvalue in $\FF_q$,
   then $|\{\tr(uv) \mid (u,v) \in A \times B\}|=q$.
\end{enumerate}
\end{lem}

\begin{proof}
This is a straightforward computation. See also Macbeath \cite{Mac}.
\end{proof}

For the rest of this section, we fix a prime power $q$. Let $S=\PSL_d(q)
\le H \le G = \PGL_d(q)$ with $d \ge 3$. Let $V$ be the natural module
for the lift of $G$ to $\GL_d(q)$. Let $f(g)$ denote the number of fixed
$1$-spaces of an element $g \in G$. Let $\chi = f -1$ and note that
$\chi$ is an irreducible character of $G$ (and $S$).

\begin{lem} \label{lem:psl}
 Let $a, b$ be nontrivial elements of $G$ and set $A=a^H$, $B=b^H$ and
 $c=ab$. If $f$ is constant on $AB$, then $f(a)$, $f(b)$, and $f(c)$
 are each at least $2$.
\end{lem}

\begin{proof}
Lift $a$ and $b$ to elements in $\GL_d(q)= \GL(V)$ (we abuse notation and
still denote them by $a$ and $b$). Note that $|\chi(g)| \geq 1$
if $\chi(g) \neq 0$.

If $f(a)=0$, then $-\chi(b) = \chi(1)\chi(c)$ by Lemma~\ref{lem:basic},
whence $\chi(c) = \chi(b) = 0$ and so each of $b$ and $c$ fixes
a unique line. Similarly, if $f(a)=1$, then $\chi(a)=0$, whence $\chi(c)=0$
and so $a$ and $c$ each fix a unique line. So we may assume that $a$ and
$c$ each fix a unique line (interchanging $a$ and $b$ if necessary).

By scaling we may assume that the unique eigenvalue of $a$ in $\FF_q$ is $1$.
Note that if both $a$ and $b$ have cyclic submodules of dimension at least $3$,
then there are $u \in A$ and $v \in B$ with $uv$ fixing at least two
lines. (Indeed, let $e_1, e_2, e_3$ be part of a basis. Then we can choose $u$
sending $e_i$ to $e_{i+1}$ for $i=1,2$ and $v$ sending
$\langle e_i \rangle_{\FF_q}$ to $\langle e_{i-1} \rangle_{\FF_q}$ for $i=2,3$.)
Then $f(uv) > 1 = f(c)$, a contradiction.

So one of $a$ or $b$ has a quadratic minimal polynomial. Note that $a$
cannot have a quadratic minimal polynomial, since its minimal polynomial
has a linear factor and it fixes a unique line. So $b$ has a quadratic minimal
polynomial. Note that as long as $d > 3$, $a$ will either contain
a $4$-dimensional cyclic submodule or a direct sum of two cyclic
submodules of dimension at least $2$. Thus, if $b$ has a $4$-dimensional
submodule that is a direct sum of two $2$-dimensional cyclic modules, as above
we can arrange that there are conjugates $u, v$ with $f(uv) > 1$.
So $b$ has no submodule that is the direct sum of two cyclic submodules
of dimension $2$. This forces $b$ to be (up to scaling) either a transvection
or a pseudoreflection for $d > 3$. The same is true for $d = 3$.

So assume that this is the case. Suppose that $a$ is not unipotent.
Write $a = a_1 \oplus a_2$ where $a_1$ is a single Jordan block
and $a_2$ fixes no lines. Conjugate $b$ so that we may write
$b = b_1 \oplus b_2$ where $b_2$ is not a scalar and $b_1$ is $1$
(and $a_i$ has the same size as $b_i$). Then
since $b_2$ has a $2$ dimensional cyclic submodule as does $a_2$, we
can arrange that $a_2b_2$ fixes a line. Thus, $f(ab) > 1$, a contradiction.

So we may assume that $a$ is a single Jordan block. If $b$ is a transvection,
then we can conjugate such that $ab$ is a unipotent element with $2$ blocks, a
contradiction.

The remaining case is where $a$ is a single Jordan block and $b$ is a
pseudoreflection. So we may assume that $a$ is upper triangular and
$b$ is diagonal. Then $ab$ will have two distinct eigenvalues in $\FF_q$, whence
$f(ab) > 1$, a contradiction.
\end{proof}

We now prove the main result of this section.

\begin{thm}   \label{thm:psl}
 Let $H=\PSL_d(q) \le G=\PGL_d(q)$ with $d \ge 3$. If $a, b$ are nontrivial
 elements of $G$, then $f$ is not constant on $a^Hb^H$.
\end{thm}

\begin{proof}
Let $m(a)$ and $m(b)$ denote the dimensions of the largest
eigenspaces (with eigenvalue in $\FF_q$) for $a$ and $b$, respectively.
Assume that $m(a) \ge m(b)$, and set $c := ab$.

If $f$ is constant on $a^Hb^H$, then $\chi(a)\chi(b) = \chi(1)\chi(c)$. We know
that $\chi(a), \chi(b)$ and $\chi(c)$ are all positive by the previous lemma.

Note that $m(c) \ge m(b)$ (since we can conjugate and assume that the largest
eigenspace of $b$ is contained in that of $a$). Note also that
$\chi(a)\le q^{d-2} + \ldots + 1$ (with equality precisely when $a$ is
essentially a pseudoreflection). Thus, $\chi(a) \le \chi(1)/q$.

First assume that $m(b) > 1$.
Then $\chi(b)<q^{m(b)}-1$ and $\chi(c)\ge q^{m(b)-1}+\ldots+q>\chi(b)/q$,
whence $\chi(a)\chi(b) < \chi(1)\chi(c)$, a contradiction. If $m(b)=1$, then
$\chi(b) \le q-1$ and $\chi(c) \ge 1$ (by the previous lemma) and we have
the same contradiction.
\end{proof}

%%%%%%%%%%%%%%%%%%%%%%%%%%%%%%%%%%%%%%%%%%%%%%%%%%%%%%%%%%%%%%%%%%%%%%%%%
\section{Classical and low rank exceptional type groups}   \label{sec:classic}

We first prove the Arad--Herzog conjecture for some low rank 
classical groups.

\begin{prop}  \label{prop:smallPSU}
 Conjecture A  holds for $\PSU_n(q)$ with $3\le n\le 6$,
 $(n,q)\ne(3,2)$.
\end{prop}

\begin{proof}
The values of the unipotent characters of $\GU_n(q)$, $3\le n\le 6$, are
contained in \Chevie{} \cite{Chevie}. Now unipotent characters restrict
irreducibly to the derived group $\SU_n(q)$, and are trivial on the center,
so can be regarded as characters of the simple group $\PSU_n(q)$. It turns
out that for
$a,b,c\in \GU_n(q)$ non-central the equation $\chi(a)\chi(b)=\chi(1)\chi(c)$
is only satisfied for all unipotent $\chi\in\Irr(\GU_n(q))$ when either
$c$ is regular unipotent, or $a$ is unipotent with one Jordan block of size
$n-1$, $b$ is semisimple with centralizer $\GU_{n-1}(q)$ (in $\PSU_n(q)$)  
and $c = xy$ is a commuting product with $x$ conjugate to $a$ and 
$y$ conjugate to $b$. In particular, in the
latter case all three classes have representatives in $\GU_{n-1}(q)$, and
it is straightforward to see that the product hits more than one class. 
The situation of the former case is ruled out by Theorem~\ref{thm:main3} 
(which does not rely on this result).
\end{proof}

\begin{prop}  \label{prop:smallclassical}
 Conjecture A  holds for $\PSp_4(q)$, $\PSp_6(q)$, $\OO_8^+(q)$
 and $\OO_8^-(q)$.
\end{prop}

\begin{proof}
The values of the unipotent characters of the conformal symplectic
group $\CSp_{2n}(q)$, $n=2,3$, of the conformal spin group $\CSpin_8^+(q)$
and of a group of type $^2D_4(q)$ are available in \cite{Chevie}. As before,
unipotent characters restrict irreducibly to the derived group and are trivial
on the center, so can be regarded as characters of the simple group
$\PSp_{2n}(q)$ respectively $\OO_8^{\pm}(q)$. Again, for given non-central
elements $a,b,c$ the equation $\chi(a)\chi(b)=\chi(1)\chi(c)$ fails for at
least one unipotent character $\chi$, unless either $c$ is regular unipotent,
which by Theorem~\ref{thm:main3} does not give rise to an example, or $n=2$, $q$ is
odd and one of $a$, $b$ is an element with centralizer $\SL_2(q^2)$.
But $\Sp_4(q)$ does not contain such elements.
\end{proof}

Unfortunately, \Chevie{} does not contain the unipotent characters of any group
related to $\OO_7(q)$.

Next we prove the Arad-Herzog conjecture for 
the low rank exceptional type groups.

\begin{prop}   \label{prop:smallexc}
 Conjecture A  holds for the groups
 $$\tw2B_2(2^{2f+1})\ (f\ge1),\ \tw2\,G_2(3^{2f+1})\ (f\ge1),\
    G_2(q)\ (q\ge3),\ \tw3D_4(q),\ \tw2F_4(2^{2f+1})\ (f\ge1).$$
\end{prop}

\begin{proof}
The generic character tables of all of the above groups $G$ are available in
the \Chevie{} system \cite{Chevie}, respectively, the values of all unipotent
characters in the case of $\tw2F_4(2^{2f+1})$. It can be checked easily
that the equation
$\chi(a)\chi(b)=\chi(c)\chi(1)$ is not satisfied simultaneously for all
unipotent characters $\chi$ of $G$, for any choice of $a,b,c\ne1$; except
when
\begin{enumerate}[(1)]
 \item $G=\tw2\,G_2(q^2)$ with $b,c$ of order dividing $q^2-1$;
 \item $G=G_2(q)$, $\gcd(q,6)=1$, with $b,c$ regular unipotent; or
 \item $G=\tw3D_4(q)$, $q$ odd, with $b,c$ regular unipotent.
\end{enumerate}
In the latter three
cases, the required equality fails on some of the two, respectively four,
irreducible characters lying in Lusztig series parametrized by an involution
in the dual group.
\end{proof}

In fact, one does not even need all the characters mentioned in the above
proof: in all cases just four of them will do. We also note that
the cases of $\PSL_3(q)$, $\PSU_3(q)$, $^2G_2(q)$ and $\PSp_4(q)$ are
handled in \cite{MoTo} using available character tables but somewhat more
elaborate arguments.
\bigskip

In the next three results, by {\it a finite classical group} we mean
any non-solvable group of the form $\SL(V)$, $\SU(V)$, $\Sp(V)$, or $\SO(V)$, 
where $V$ is a finite vector space. First we note (see also \cite{law}):

\begin{lem}   \label{lem:unipotent pairs}
 Let $G$ be a finite classical group with natural module $V$ of dimension
 $d$ over the finite field $\FF_q$.  Assume that $G$ has rank at least $2$
 and that $\dim V \ge  6$ if $G$ is an orthogonal group. 
 Let $x\in G$ be a nontrivial unipotent element of $G$.
 \begin{enumerate}[\rm(a)]
  \item Let $P$ be the stabilizer of a singular $1$-space with $Q$ the
   unipotent radical of $P$. If $x^G \cap P \subseteq Q$, then either
    $G=\Sp_4(q)$ with $q$ even and $x$ is a short root element, or
    $G=\SU_4(q)$.
  \item If $d=2m$ and either $G=\Sp_{d}(q)$ with $q$ even or $G=\SU_{d}(q)$,
   and $P$ is the stabilizer of a maximal totally isotropic subspace with $Q$
   the unipotent radical of $P$ and $x^G\cap P \subseteq Q$, then $x$ is a
   long root element.
 \end{enumerate}
\end{lem}

\begin{proof}
Consider (a). By assumption $x$ is conjugate to an element of $Q$. If
$G=\SL_d(q)$, this forces $x$ to be a transvection and the result is clear.
Otherwise, we have $\dim (x-1)V \le 2$. 

If $d \le 4$, this is a straightforward computation (in particular, for $q$
even, all short root elements in $P$ are contained in $Q$).   

If $d > 4$, then $x$ will act trivially on a nondegenerate space.  Thus, 
if $G=\SU_d(q)$, it suffices to show the claim for $d=5,6$ and this is a 
straightforward computation. In all other cases, it follows by the results
for $d \le 4$ noting that if  $x$ is a short root element, then clearly $x$
is conjugate to an element in a Levi subgroup of $P$.

Now consider (b) with $G=\Sp_{d}(q)$.  Again, we may assume that $x \in Q$
and $x$ is not a transvection.
We can identify $Q$ with the set of symmetric matrices of size $m$.
Since $x$ is not a transvection,  $x$ corresponds to a symmetric
matrix of rank at least $2$. If $x$ corresponds to a skew symmetric
matrix, then we see that $V = V_1 \perp V_2$ where $V_1$ is $4$-dimensional
and $x$ is a short root element on $V_1$ and the result follows by induction.
If $x$ does not correspond to a skew symmetric matrix, we may conjugate
$x$ so that it corresponds to a diagonal matrix of rank at least $2$,
whence we see that $V = V_1 \perp V_2$ where $V_1$ is $4$-dimensional
and $x$ has two Jordan blocks on $V_1$ each nondegenerate.
A straightforward computation shows that $x$ stabilizes
and acts nontrivially on a $2$-dimensional totally singular
subspace of $V_1$ and so also on $V$.   

If $G=\SU_{d}(q)$, we can identify $Q$ with Hermitian $m \times m$ matrices and
every element of $Q$ is conjugate to a diagonal element. Since $g \in Q$ is
nontrivial and not a transvection, it corresponds to an element of rank
at least $2$ in $Q$ and a straightforward computation in $\SU_4$ gives the
result. 
\end{proof}

We can use this to prove the following result about pairs of unipotent
elements in classical groups.

\begin{thm}  \label{thm:unipotent prods}
 Let $G$ be a finite classical group with natural module $V$ of dimension
 $d  \ge  2$ over the finite field $\FF_q$. Let $x, y \in G$ be nontrivial
 unipotent elements of $G$. Then one of the following holds:
 \begin{enumerate}[\rm(1)]
  \item $x^Gy^G$ does not consist of unipotent elements; or
  \item $G=\Sp_d(q) = \Sp(V)$, $d\ge 4$, with $q$ even and (up to order) $x$
   is a long root element and $y$ is an involution such that $(yv,v) = 0$
   for all $v \in V$.
%%  \item  $G=\Sp_d(3)$ and $x$ and $y$ are conjugate long root elements.
 \end{enumerate}
\end{thm}

\begin{proof}   
First exclude the case that $G=\Sp_d(q)$, $d\ge 4$, with $q$ even or
$G=\SU_{d}(q)$ with $d \ge 4$ even. Let
$P$ be the stabilizer of a singular $1$-space. By~(a) of the previous result
and induction, we are reduced to considering $G=\SL_2(q)$
and $G=\SU_3(q)$.
We can then apply Lemma \ref{lem:d=2} (and just compute to see that this
is still true for $q \le 3$) and similarly for $\SU_3(q)$.

Next consider $G=\Sp_d(q)$ with $q$ even. Suppose that neither $x$ nor $y$ is
a long root element. By applying~(b) of the previous lemma, we are reduced
to the case of $\SL_m(q)$ where $d=2m$. If $x$ and $y$ are both
long root elements, the result is clear (even for $q= 2$) by reducing
to the case of $\SL_2$. So we may assume that
$x$ is a long root element and that $y$ is either not an involution
or $(yv,v) \ne 0$ for some $v \in V$. Indeed, in either case there exists
$v \in V$ with $(yv,v) \ne 0$. By replacing $x$ by a conjugate, we may assume
that $x$ leaves $W:=\langle v, yv\rangle$ invariant and acts nontrivially.
Writing $V = W \oplus W^{\perp}$ and conjugating $x$ on $W$ as necessary, it
is an easy linear algebra computation to see that we can arrange for
$\tr(xy) \ne d$, whence $xy$ is not unipotent.

Finally, consider $G=\SU_{d}(q)$ with $d$ even.  If $d=2$, then the result follows by the
case of $\SL_2(q)$.  By applying~(b) of the previous lemma, we see that
we may assume that $x$ is a long root element.  By applying~(a) of the previous
lemma, we are reduced to the case of $\SU_4(q)$ and a straightforward computation
completes the proof.
\end{proof}

Note that, by choosing $x,y$ in the same Sylow subgroup of $G$, we see that
$x^Gy^G$ always contains unipotent elements. Furthermore,
in (2) above, $x^Gy^G$ will in fact consist of unipotent elements
(since $y$ acts trivially on a maximal totally singular space, we see that
$x$ and $y$ always act trivially on a common totally singular space $U$ of
dimension $d/2 -1$ and $y$ will act trivially on the two-dimensional
space $U^{\perp}/U$, whence $xy$ is unipotent). On the other hand,
it is also straightforward to compute that $xy$ can be an involution or an
element of order $4$, whence $x^Gy^G$ is not a single conjugacy class.
Thus:

\begin{cor}   \label{cor:unipotent prods}
 Let $G$ be a finite classical group with natural module $V$ of dimension
 $d \ge 2$ over the finite field $\FF_q$.   Let $H$ be the derived
 subgroup of $G/Z(G)$.  
 Let $x,y\in  H$ be nontrivial unipotent elements of $H$.
 Then $x^Hy^H$ is not a single conjugacy class of $H$.
\end{cor}

%%%%%%%%%%%%%%%%%%%%%%%%%%%%%%%%%%%%%%%%%%%%%%%%%%%%%%%%%%%%%%%%%%%%%%%%%
\section{Semisimple and unipotent classes}   \label{sec:exc}

In this section we prove Theorem~\ref{thm:main2} (assuming a result
on algebraic groups, Corollary \ref{cor:factorize}, which is independent
of this section), and complete the proof of Theorem~\ref{thm:main3}.

First we set up some notation. Throughout this section, let $\bG$ be a
connected reductive algebraic group in characteristic $p>0$ and
$F:\bG\rightarrow \bG$ a Steinberg endomorphism of $\bG$, with (finite) group
of fixed points $G:=\bG^F$. Note that if $\bG$ is simple of adjoint type
then $S=O^{p'}(G)$ is almost always simple.
We may abuse notation and write $G=\bG(q)$ where $q$ is the power of $p$
(always integral unless $G$ is a Suzuki or Ree group; in the latter case we
write $\bG(q^2)$ instead) such that $F$ acts as $q\phi$ on the character group
of an $F$-stable maximal torus of $\bG$, with $\phi$ of finite order.
Note that if $a \in G$ is semisimple, then $a^G=a^S$ (see \cite[2.12]{seitz}).

\subsection{Proof of Theorem \ref{thm:main2}}
Theorem \ref{thm:main2} follows from the following, slightly more general result:

\begin{prop}   \label{prop:regss}
 Let $S \le H \le G$. Let $a,b\in H$ be nontrivial semisimple elements.
 Then the Steinberg character is not constant on $a^Sb^S$. In particular,
 $a^Hb^H \ne c^H$ for any $c \in H$.
\end{prop}

\begin{proof}
Let $\St$ denote the Steinberg character of $G$.
Note that $\St$ restricts irreducibly to $S$ unless
$G={}{^2}G_2(3), G_2(2), \Sp_4(2)$ or $\tw2F_4(2)$. In those
cases, one can verify the result directly (in the last case,
we could use the two ``half-Steinberg'' representations).
Note that if $g \in G$, then
$$\St(g)=\begin{cases}\pm|C_G(g)|_p=\pm q^{m(g)}& \text{if $g$ is semisimple,}\\
  0& \text{else,}\end{cases}$$
where $m(g)$ is the dimension of a maximal unipotent subgroup of $C_\bG(g)$
(see for example \cite[Thm.~6.4.7]{Ca}). In particular,
$\St(1)=q^N$ where $N$ is the number of positive roots of $\bG$.

Suppose that $a$ and $b$ are nontrivial semisimple elements
and $\St$ is constant on $a^Hb^H$. Then $\St(a)\St(b)=\St(c)\St(1)$
for $c := ab$, by Lemma~\ref{lem:basic}. In particular, $\St(c) \ne 0$,
whence $c$ is also semisimple. This in turn implies that
$m(a)+m(b) = m(c)+N$.

Since $C_\bG(a)$ is reductive and contains a maximal torus of $\bG$,
we see that $\dim C_\bG(a) = 2m(a) + r$
where $r$ is the rank of $\bG$ (and similarly for $b$ and $c$). Thus,
$$\begin{aligned}
  \dim C_\bG(a) + \dim C_\bG(b) &=2(m(a) + m(b)) + 2r = 2r +  2m(c) + 2N \\
  &=(r + 2N) + (r+ 2m(c)) = \dim \bG + \dim C_\bG(c).
\end{aligned}$$
Let $f:C_\bG(a) \times C_\bG(b) \rightarrow \bG$ be the multiplication map.
Note that each fiber has dimension equal to
$\dim (C_\bG(a)\cap C_\bG(b))$ which is at most $\dim C_\bG(c)$ as $c = ab$.
It follows that
$$\begin{aligned}
  \dim C_\bG(a)C_\bG(b)&=\dim C_\bG(a)+\dim C_\bG(b)-\dim (C_\bG(a)\cap C_\bG(b))\\
  &\ge \dim C_\bG(a) + \dim C_\bG(b) -  \dim C_\bG(c) = \dim \bG.
\end{aligned}$$
Thus, $C_\bG(a)C_\bG(b)$ is dense in $\bG$. By Corollary~\ref{cor:factorize}
(below) this cannot occur, so $\St$ cannot be constant on $a^Sb^S$.
\end{proof}

\subsection{Proof of Theorem \ref{thm:main3}}
For any $F$-stable maximal torus $\bT$ of $\bG$ and any $\theta\in\Irr(\bT^F)$,
Deligne and Lusztig defined a generalized character $R_{\bT,\theta}^\bG$
of $G = \bG^F$. Its restriction $Q_\bT^\bG:=R_{\bT,\theta}^\bG|_{G_u}$ to the
set $G_u$ of unipotent elements of $G$ is independent of $\theta$, rational
valued, and called the Green function corresponding to $\bT$ (see for instance
\cite[\S\S7.2, 7.6]{Ca}).

The following is an easy consequence of Deligne--Lusztig's character formula
for $R_{\bT,\theta}^\bG$:

\begin{prop}   \label{prop:charval}
 Let $\bG,F$ be as above. Let $x\in G$, with Jordan decomposition
 $x=su$, where $s\in G$ is semisimple, $u$ is unipotent.
 Let $\bT\le\bG$ be an $F$-stable maximal torus with $s\notin\bT^g$ for all
 $g\in G$. Then $R_{\bT,\theta}^\bG(x)= 0$ for any $\theta\in\Irr(\bT^F)$.
\end{prop}

\begin{proof}
By \cite[Thm.~7.2.8]{Ca} we have
$$R_{\bT,\theta}^\bG(x)
  =\frac{1}{|\bC^F|}\sum_{g\in G\,:\,s^g\in \bT^F}\theta(s^g)\,Q_{^g\bT}^\bC(u),$$
where $\bC=C_\bG^\circ(s)$. Clearly, this shows that $R_{\bT,\theta}^\bG(x)=0$
unless $s\in {}^g\bT^F$ for some $g\in G$.
\end{proof}

Now assume that $S=G/Z(G)$ is a finite simple group (which usually happens when
$\bG$ is simple of simply connected type).

\begin{prop}   \label{prop:uni}
 In the above setting, let $c\in S$ be unipotent and suppose that there are
 $a,b\in S$ with $a^S b^S=c^S$. Then for any $F$-stable maximal torus
 $\bT\le\bG$ such that $T=\bT^F$ has a character $\theta$ in general position
 with $\theta|_{Z(G)}=1$ we
 have:
 \begin{enumerate}[\rm(a)]
  \item the semisimple parts $a_s,b_s$ of $a,b$ are conjugate to elements
   of $T$, and
  \item $|C_G(a)|\,|C_G(b)|\ge |G:T|_{p'}^2$.
 \end{enumerate}
\end{prop}

\begin{proof}
As $\theta$ is in general position, $R_{\bT,\theta}^\bG$ is an irreducible
character of $G$ up to sign, say $\chi$, by \cite[Cor.~7.3.5]{Ca}. Since
$\theta|_{Z(G)}=1$ we also have $\chi|_{Z(G)}=1$ by the character formula
\cite[Thm.~7.2.8]{Ca}, so $\chi$ can be considered as an irreducible character
of $S=G/Z(G)$. Moreover, since $Q_\bT^\bG$ is rational valued
%by a well-known congruence of character values
we have
$$R_{\bT,\theta}^\bG(c)=Q_\bT^\bG(c)\equiv Q_\bT^\bG(1)=R_{\bT,\theta}^\bG(1)
  =\pm|G:T|_{p'}\equiv\pm1\pmod p.$$
(Here, the first congruence holds for any generalized character of the cyclic
$p$-group $\langle c \rangle$, and the second congruence holds since it is true
for cyclotomic polynomials in $q$.)
In particular, $\chi(c)$ is a nonzero integer.
Thus, by Lemma~\ref{lem:basic}, we also have
$\chi(a)\chi(b)=\chi(1)\chi(c)\ne0$. By Proposition~\ref{prop:charval} this
gives~(a), and moreover
$$|C_G(a)|\cdot|C_G(b)|\ge|\chi(a)|^2\cdot|\chi(b)|^2\ge \chi(1)^2 = |G:T|_{p'}^2.$$
\end{proof}

\begin{prop}  \label{prop:unipotent prods}
 Let $S$ be a simple group of Lie type, and $a,b,c\in S\setminus\{1\}$ such
 that $a^Sb^S=c^S$. Then $c$ is unipotent if and only if both $a$ and $b$ are.
\end{prop}

\begin{proof}
Let $\bG$ be a simple, simply connected algebraic group over an algebraic
closure of $\FF_p$ and $F:\bG\rightarrow\bG$ a Steinberg endomorphism whose
group of fixed points $G=\bG^F$ satisfies $S=G/Z(G)$. This is possible unless
$S={}\tw2F_4(2)'$, for which the claim is easily checked directly. If $S$ is
of exceptional
type and twisted Lie rank at most~2, the claim has already been proved in
Proposition~\ref{prop:smallexc}. For all other types we have given in
Tables~\ref{tab:exctori} and~\ref{tab:classtori} two maximal tori of $G$
(see \cite[Tables~5.2 and~5.8]{MNO}), with the following properties: in the
exceptional types always, and in the classical types whenever the corresponding
Zsigmondy primes $\ell_i$ exist, the dual tori contain regular elements of
order this Zsigmondy prime (by
\cite[Lemmas~5.3 and 5.9]{MNO}) and with
connected centralizer in the dual group. Note that $\ell_i$ is
coprime to $|Z(G)|$, so both tori have characters $\theta_i$ in
general position with the center in their kernel. \par
Now let $\tilde a,\tilde b$ be preimages of $a$, $b$ respectively
in $G$. Assume that $\tilde c\in G$ is unipotent, with image $c$ in $S$.
Then Proposition~\ref{prop:uni} applies to say that $\tilde a_s$ is conjugate
to elements of both $T_1,T_2$. But in all cases the intersection of $T_1$ with
any conjugate of $T_2$ lies in $Z(G)$, so $a$ is unipotent, and similarly
for $b$. \par
We now consider the classical groups for which not both Zsigmondy primes
exist. The groups $\PSL_2(q)$, $\PSL_3(q)$, $\PSL_6(2)$, $\PSL_7(2)$ are
handled in Lemma~\ref{lem:d=2} and Theorem~\ref{thm:psl}. For the unitary
groups $\PSU_3(q)$ and the symplectic groups $\PSp_4(q)$ as well as for the
groups $\PSU_6(2)$, $\PSp_6(2)$, $\OO_8^+(2)$, $\OO_8^-(2)$,
the claim follows by Propositions~\ref{prop:smallPSU}
and~\ref{prop:smallclassical} while for the groups $\PSU_7(2)$, $\PSp_8(2)$
it can be checked directly using the character tables in GAP. \par
Conversely, if $a,b$ are unipotent then without loss they lie in a common Sylow
$p$-subgroup of $S$, and hence so does $c$, whence it is unipotent.
\end{proof}

\begin{table}[htbp]
 \caption{Two tori and Zsigmondy primes in exceptional groups}
  \label{tab:exctori}
\[\begin{array}{|r||l|l|l|l|}
\hline
     G& |T_1|& |T_2| & \ell_1 & \ell_2 \cr
\hline
     F_4(q)&       \Phi_8&          \Phi_{12} & l(8) & l(12) \cr
     E_6(q)&       \Phi_9&          \Phi_1\Phi_2\Phi_8 & l(9) & l(8) \cr
 \tw2E_6(q)&       \Phi_{18}&       \Phi_1\Phi_2\Phi_8 & l(18) & l(8) \cr
     E_7(q)&       \Phi_2\Phi_{18}& \Phi_1\Phi_7 & l(18) & l(7) \cr
     E_8(q)&       \Phi_{30}&       \Phi_{24} & l(30) & l(24) \cr
\hline
\end{array}\]
\end{table}

\begin{table}[htbp]
 \caption{Two tori and Zsigmondy primes in classical groups}
  \label{tab:classtori}
\[\begin{array}{|rl||l|l|l|l|}
\hline
     G& & |T_1|& |T_2|& \ell_1& \ell_2\cr
\hline
 A_n& & (q^{n+1}-1)/(q-1)& q^n-1& l(n+1)& l(n)\cr
 \tw2A_n& (n\ge2$ even$)& (q^{n+1}+1)/(q+1)& q^n-1& l(2n+2)& l(n)\cr
 \tw2A_n& (n\ge3$ odd$)& (q^{n+1}-1)/(q+1)& q^n+1& l(n+1)& l(2n)\cr
 B_n,C_n& (n\ge2$ even$)& q^n+1& (q^{n-1}+1)(q+1)& l(2n)& l(2n-2)\cr
 B_n,C_n& (n\ge3$ odd$)& q^n+1& q^n-1& l(2n)& l(n)\cr
 D_n& (n\ge4$ even$)& (q^{n-1}-1)(q-1)& (q^{n-1}+1)(q+1)& l(n-1)& l(2n-2)\cr
 D_n& (n\ge5$ odd$)& q^n-1& (q^{n-1}+1)(q+1)& l(n)& l(2n-2)\cr
 \tw2D_n& (n\ge4)& q^n+1& (q^{n-1}+1)(q-1)& l(2n)& l(2n-2)\cr
\hline
\end{array}\]
\end{table}

Together with Corollary~\ref{cor:unipotent prods} this establishes
Theorem~\ref{thm:main3} for classical groups. To complete the proof
of Theorem~\ref{thm:main3} for exceptional groups, we need the following
result of Lawther:

\begin{lem}[Lawther]   \label{f4 lemma}
 Let $G=F_4(q)$ with $q$ even. Let $P=QL$ be a maximal end node parabolic with
 unipotent radical $Q$ and Levi subgroup $L\cong C_3(q)T_1$.
 If $u \in G$ is a nontrivial unipotent element such that
 $u^G \cap P \subset sQ \cup Q$ where $s$ is a long root element in $L$,
 then $u$ is a long root element.
\end{lem}

\begin{proof}
The proof is a case by case analysis.
Write roots in $F_4$ as linear combinations of simple roots, so that for
example the highest root is denoted $2342$. Write $w_i$ for the Weyl group
reflection corresponding to the $i$th simple root.

Let us say that if $x$ is a product of positive root elements,
at least one of whose roots is in $\{ 0010, 0001, 0011 \}$, then $x$ has
property (*). Observe that if $x$ has property (*), then $x$ mod $Q$ is
neither the identity nor a long root element of $L$. Now note that Shinoda
\cite[p.~130]{shinoda} has listed unipotent class representatives
$x_0,x_1,\dots,x_{34}$ of $G$.
Recall that $x_0 = 1$ and $x_2$ is a long root element. Thus it suffices to
observe that for $i = 1,3,4,\dots,34$ there is a $g_i \in G$ so
that $g_i^{-1}x_ig_i \in P$ has property (*).

If $i=1, 3, 4$, take $g_i =w_4w_3w_2w_1w_3w_2$. If $5 \le i \le 16$, take
$g_i = w_1w_2$. If $19 \le i \le 21$, take $g_i=w_2$. In the remaining cases,
$x_iQ$ is neither trivial nor a long root element. The result follows.
\end{proof}

Now we can use our methods together with another result 
of Lawther \cite{law} to obtain the following:

\begin{thm}   \label{bs1}
 Let $G$ be a finite simple group of Lie type in characteristic $p$.
 Let $u, w$ be nontrivial unipotent elements of $G$. 
 Then $u^Gw^G$ is not a single conjugacy class.  If $uw^g$ is unipotent
 for all $g \in G$, then $p \le 3$ and (up to order) one the following holds:
 \begin{enumerate}[\rm(1)]
  \item $G=\Sp_{2n}(q)$, $p=2$, $u$ is a long root element and $w$ is an
   involution (which satisfies $(wv,v)=0$ for all $v\in V$, the natural module);
  \item $G=F_4(q)$, $p=2$, $u$ is a long root element and $w$ is a short
   root element; or
  \item $G=G_2(q)$, $p=3$, $u$ is a long root element and $w$ is a short
   root element.
 \end{enumerate}
\end{thm}

\begin{proof}
If $G$ is classical, this follows by Theorem \ref{thm:unipotent prods}.
If $G={^2}G_2(q^2)$, $\tw2B_2(q^2)$, $\tw2F_4(q^2)'$ or $\tw3D_4(q)$, the result
follows by a computation using \Chevie. If $G=E_n(q)$, then by \cite[\S2]{GuSa},
we can assume that $u,w$ are in an end node parabolic subgroup
and not in its radical. The result now follows by induction (since none of
the exceptions occur in the inductive step).
If $G=\tw2E_6(q)$, then by Lawther \cite{law}, the same argument applies.

Suppose that $G=G_2(q)$.  If $q=2$, one computes directly.  
Let $P_i$, $i = 1,2$ denote the two maximal parabolic subgroups containing
a fixed Borel subgroup.  Let $Q_i$ be the unipotent radical of $P_i$. 
If $p \ne 3$ and $q > 2$, it follows by \cite{law} that any unipotent element
is conjugate to an element of  $P_1\setminus{Q_1}$
and so the result follows by the result for $A_1$.   If  $p=3$, then also by
\cite{law} unless $u$ is  a long root element and $w$ is a short root element
(or vice versa), 
$u, w$ are conjugate to elements in $P_i \setminus{Q_i}$ for $i=1$ or $2$
and the result follows by the case of $A_1$. 
Alternatively, one can compute using \Chevie.  

It remains to consider $G=F_4(q)$. Let $P$ be a maximal parabolic
subgroup with Levi subgroup of type $B_3(q)$. By \cite{law},
we may assume that $u, w$ are in $P$ and not in the radical $Q$ of $P$.
Arguing as above, we may reduce to the case of $B_3(q)$, whence the
result for $q$ odd.
If $q$ is even, the same argument shows that the result holds unless (up to
order), $u^G \cap P \subset Q \cup xQ$ where $x$ is a long root element and
$w^G \cap P \subset Q \cup yQ$ where $y$ is a short root element.
Now Lemma~\ref{f4 lemma} forces $u$ to be a long root element. Now replace
$u$ and $w$ by their images $u'$ and $w'$ under the graph automorphism.
So $u'$ is a short root element. As above, this forces $w'$ to be a long
root element, whence $w$ is a short root element.

We now show that $u^Gw^G$ is not a single conjugacy class. If so, then $uw^g$
is conjugate to $uw$ for all $g$.  Of course, $uw^g$ may be unipotent.
So the result is clear aside from the three special cases above.  In (1), it
is straightforward to observe that $uw^g$ may have order either $2$ or $4$.
Consider (2).  Since we can choose $u, w^g \in H \le F_4$ with
$H \cong \Sp_4(q)$,  the result holds.  Finally, in (3), it is straightforward
to see that $uw^g$ can be a regular unipotent element (and so of order $9$).
On the other hand, $u$ and $w$ are both conjugate to central elements in a
Sylow $3$-subgroup, whence $uw^g$ can also be a $3$-central element
(of order $3$). This completes the proof.
\end{proof}

In fact, we will see (in Examples~\ref{ex: G2}, \ref{ex:f4},
and~\ref{ex: sp uni}) that in all the exceptional cases in the previous 
result, $\langle u, w^g \rangle$ is unipotent for all $g$ (even in the
corresponding algebraic group).

Lawther~\cite{law} proves much more than we require for the proof of
Theorem~\ref{bs1}.  He determines all pairs of conjugacy classes $C$ of
unipotent elements and maximal parabolic subgroups $P$ of a finite simple
group of Lie type such that $C \cap P$ is contained in the unipotent radical
of $P$. 
 
Now Theorem~\ref{thm:main3} immediately follows from Proposition 
\ref{prop:unipotent prods} and Theorem \ref{bs1}.

\subsection{Some permutation characters} 
We now prove some results on certain permutation characters for
$G_2(q)$ and $F_4(q)$ that we will need for our results on algebraic groups.

\begin{lem}  \label{lem:g2not3}
 Let $G=G_2(q)$ with $(q,3)=1$. Let $a$ be a long root element and $b$ an
 element of order $3$ with centralizer $\SL_3(q)$ or $\SU_3(q)$ (depending
 upon whether $q \equiv 1 \pmod3$ or not).
 Let $C=C_G(a)$ and $D=C_G(b)$. Then the scalar product $[1_C^G, 1_D^G]$ equals
 $2$. Moreover, if $q \equiv 1 \pmod3$, then $\langle a, b \rangle$ is
 contained in a Borel subgroup of $G$.
\end{lem}

\begin{proof}
We give the proof for $q \equiv 1 \pmod3$. Essentially the identical proof
works in the other case. Moreover, for our application to algebraic groups,
this case is sufficient.

Note that $[1_C^G, 1_D^G] =   |C \backslash G /D|$  
or equivalently the number of orbits of $G$ on $\Gamma: =a^G \times b^G$
(acting by simultaneous conjugation).

We will produce two distinct $G$-orbits on $\Gamma$ and show that the number
of elements in the union of these orbits is $|\Gamma|$, whence the result.

The first orbit consists of the commuting pairs in $\Gamma$.
We can conjugate and assume that the second element is $b$ and so $a$ must be
a long root elements in $D$. We thus see that this is a single orbit of
size $q^3(q^3+1)(q+1)(q^3-1)$.

Using \Chevie, we see that we may choose $(c,d) \in \Gamma$
such that $cd$ is conjugate to $bu$
 with $u$ a regular unipotent element in $D$.
Thus, $C_G(c) \cap C_G(d)$ is isomorphic to a subgroup of $C_D(u)$ which
has order $3q^2$. We claim that $C_G(c) \cap C_G(d)$ contains no elements
of order $3$. This is because the only elements of order $3$ in $D$ which
are conjugate to $b$ in $G$ are $b$ and $b^{-1}$. Since $c$ does not commute
with $d$, it follows that no element of order $3$ is in $C_G(c) \cap C_G(d)$.
Thus, $|C_G(c) \cap C_G(d)| \le q^2$ (in fact, we have equality but this
will come out).

Thus, the size of the $G$-orbit containing $(c,d)$
is $[G:(C_G(c) \cap C_G(d))] \ge q^4(q^2-1)(q^6-1)$.
It follows that the size of the union of these two
orbits is at least $|\Gamma|$ (and so exactly).

Since we are assuming that $q \equiv 1 \pmod3$, $b$ is contained in some
Borel subgroup $B$ of $G$ containing the Borel subgroup of $C_G(b)$. Let $T$
be a maximal torus of $C_B(b)$ (and so also of $G$). Let $a_1$ be a long
root element of $C_B(b)$. Let $J$ be the subgroup of $B$ generated by $T$
and all long root elements of $B$. Since $J$ is normal in $B$ and $C_B(b)$
is not normal in $B$, we can choose a long root element $a_2$ of $B$ not
in $C_B(g)$. Thus, $(a_1,b)$ and $(a_2,b)$ are in different $G$-orbits on
$a^G \times b^G$. It follows that each pair in $\Gamma$ is contained in
some Borel subgroup of $G$.
\end{proof}

\begin{lem}   \label{lem:f43a}
 Let $G=F_4(q)$ with $q$ odd. Let $a$ be a long root element of $G$ and $b$
 an involution in $G$ with centralizer $H :=C_G(b)$ of type $B_4(q)$.
 Let $P$ be the normalizer of the long root subgroup of $G$ containing~$a$,
 so that $P' = C_G(a)$. Then $[1^G_H,1_{P'}^G] = 2$.
 Moreover, if $(c,d) \in a^G \times b^G$, then $\langle c, d \rangle$ is
 contained in a Borel subgroup of $G$.
\end{lem}

\begin{proof}
Certainly, $1_{P'}^G = \sum_{\lambda \in \Irr(P/P')}\lambda^G$, with
$P/P'\cong C_{q-1}$.

Let $\bP\le\bG$ be an $F$-stable parabolic subgroup with $\bP^F=P$, and $\bL$
an $F$-stable Levi subgroup of $\bP$. Any nontrivial linear character
$\lambda$ of $P/P'$ can be viewed as a linear character of $L=\bL^F$, and
then $\lambda^G$ is the Harish-Chandra induction $R_\bL^\bG(\lambda)$ of
$\lambda$. Thus $\lambda$ belongs to the Lusztig series $\cE(L,s)$, where
$s$ is a nontrivial central (semisimple) element of $L^*\le G^*={\bG^*}^{F^*}$,
the dual of $L$, where $\bG^*$ denotes the dual group (which is isomorphic
to $G$).
Now $\bL$ has type $C_3T_1$, with $T_1$ a $1$-dimensional torus. So the
underlying algebraic group $\bL^*$ with $L^* = {\bL^*}^{F^*}$ has type $B_3T_1$.
By \cite[Prop.~3.6.8]{Ca} we have $Z(\bL^*)^{F^*} = Z(L^*)$, whence
$C_{\bG^*}(s)$ contains the reductive subgroup $\bL^*$ of type $B_3T_1$.
%%Since $Z(\bG)$ is connected
Note that Lusztig induction $R_\bL^\bG$ sends any
irreducible character in $\cE(L,s)$ to a linear combination
of irreducible characters in $\cE(G,s)$, cf. for instance
\cite[Lemma 8.2]{Lub}. 
So all the irreducible
constituents $\varphi$ of $\lambda^G$ belongs to $\cE(G,s)$.

On the other hand, since $q$ is odd, by \cite[p.~110]{law3} we have
$$1^G_H = \chi_{\phi_{1,0}} + \chi_{\phi''_{8,3}} + \chi_{\phi_{4,1}}+
   \chi_{\phi''_{2,4}}
  + \chi_{\kappa_1}^{1,St} + \sum^{(q-3)/2}_{j=1}\chi^1_{\kappa_7,j}
  + \sum^{(q-1)/2}_{j=1}\chi^1_{\kappa_{8,j}},$$
where the first four constituents are unipotent characters
(and $\chi_{\psi}$ is the unipotent character labeled by the Weyl group
character $\psi$ listed in \cite[\S13.9]{Ca}). Furthermore,
the fifth constituent belongs to $\cE(G,\kappa_1)$, where
$\kappa_1 = (t_1)^{G^*}$ is the conjugacy class of an involution $t_1 \in G^*$
with $C_{\bG^*}(t_1)$ of type $C_3A_1$. Each of the summands in the next
two summations belongs to $\cE(G,\kappa_{7,j})$ or $\cE(G,\kappa_{8,j})$,
where $\kappa_{a,j} = (t_{a,j})^{G^*}$ is the conjugacy class of a
semisimple element $t_{a,j} \in G^*$, with the semisimple part of
$C_{\bG^*}(t_{a,j})$ being of type $C_3$ for $a = 7,8$. Since
$C_{\bG^*}(s)$ contains a reductive subgroup of type $B_3T_1$, $s$ cannot be
conjugate to any of the elements $1$, $t_1$, or $t_{a,j}$, $a = 7,8$.
It follows that $[1^G_H,\lambda^G] = 0$ for $\lambda\ne1$.

Thus $[1^G_H,1_{P'}^G] = [1^G_H,1_{P}^G]$, and it remains to consider the
case $\lambda = 1_{P}$. It is well known that the decomposition of $1^G_P$ into
irreducible constituents is given by the corresponding decomposition for the
permutation character of the Weyl group $W(F_4)$ acting on the cosets of the
parabolic subgroup $W(C_3)$, the Weyl group of $L$.
The irreducible constituents in the latter decomposition are
$\chi_{\phi_{1,0}}$, $\chi_{\phi_{2,4}'}$, $\chi_{\phi_{9,2}}$,
$\chi_{\phi_{4,1}}$, and $\chi_{\phi_{8,3}'}$.
Thus, the scalar product of the two permutation characters is $2$ as claimed.

It follows that $G$ has two orbits on $a^G \times b^G$. Let $B$ be a Borel
subgroup of $G$ containing a Borel subgroup of $C_G(b) \cong B_4(q)$.
Arguing as in the previous case, we can choose long root elements
$a_1, a_2 \in B$ with $a_1b=ba_1$ and $a_2b \ne ba_2$.
Certainly, $(a_1,b)$ and $(a_2,b)$ belong to different $G$-orbits on
$a^G \times b^G$. It follows that each pair in $a^G \times b^G$ is 
contained in some Borel subgroup of $G$.
\end{proof}

%%%%%%%%%%%%%%%%%%%%%%%%%%%%%%%%%%%%%%%%%%%%%%%%%%%%%%%%%%%%%%%%%%%%%%%%%
\section{Algebraic Groups}   \label{sec:alg}

We first recall some facts about conjugacy classes in algebraic groups.
Throughout the section 
we fix an algebraically closed field $k$ of characteristic $p \ge 0$.

By a fundamental result of Lusztig there are only finitely many conjugacy
classes of unipotent elements in a connected reductive group. 
This is easily seen to imply that if $A$ and $B$ are conjugacy classes of
a simple algebraic group, then $AB$ is an infinite union of conjugacy
classes if and only if the closure of $AB$ contains infinitely many
semisimple conjugacy classes. We will not use this result in what follows.

We will use the following elementary result.  Note that if $a$ is an 
element of a connected reductive algebraic group $\bG$ and $a=su=us$ where $s$
is semisimple and $u$ is unipotent, then $s \in \overline{a^{\bG}}$. 

\begin{lem}   \label{lem:irred}
 Let $\bG$ be a connected reductive algebraic group over $k$, $\bT$ a
 maximal torus of $\bG$, and let $A$ and $B$ be non-central conjugacy classes
 of $\bG$. Then the following statements hold.
 \begin{enumerate}[\rm (a)]
  \item  $\overline{AB}$ either contains a unique semisimple conjugacy class
   of $\bG$ or contains infinitely many semisimple classes.
  \item  $\overline{AB}$ contains a unique semisimple conjugacy class if
   and only if $\overline{AB} \cap \bT$ is finite.
 \end{enumerate}
\end{lem}

\begin{proof}
Suppose that $\overline{AB}$ contains finitely many
semisimple classes $C_1,\ldots, C_m$.   Let $X_i$ be the set of
elements in $\bG$ whose semisimple parts are in $C_i$.   Note that
$X_i$ is closed (since if $s \in X_i$ is a semisimple element, then
$X_i$ consists of all elements $g \in \bG$ with $\chi (g) = \chi(s)$ for
all the characters of rational finite-dimensional $\bG$-modules). 
Since $A$ and $B$ are irreducible varieties, so is $\overline{AB}$, whence
$\overline{AB} \subset \cup_i X_i$ implies that  $\overline{AB} \subset X_i$
for some $i$.    This proves (a).  

Now (b) follows by (a) and the facts that every semisimple class of $\bG$
intersects $\bT$ nontrivially and this intersection is finite 
(since it is an orbit of the Weyl group on $\bT$, see \cite[Prop. 3.7.1]{Ca}). 
\end{proof}

We need some results about closures of unipotent classes. These can be deduced
from the results in \cite{Spa}. We give elementary proofs for what we need
(but quote \cite{Spa} for $G_2$ and also for $F_4$ in characteristic $2$).
We also do not consider the groups of type $B$ in characteristic $2$.
The results in this case can be read off from the results for the groups
of type $C$. The first such result we need has a very short proof,
see \cite[Cor.~3.3]{GM}. 

\begin{lem}   \label{lem:closures1}
 Let $\bG$ be a simple algebraic group over an algebraically closed field $k$
 of characteristic $p \ge 0$ and $g \in \bG$ a nontrivial unipotent element.
 Then the closure of $g^{\bG}$ contains root elements.  
 \end{lem} 

We next note the following fact:

\begin{lem}   \label{lem:dense centralizer}
 Let $\bG$ be a semisimple algebraic group with $a, b \in \bG$.
 If $C_{\bG}(a)C_{\bG}(b)$ is dense in $\bG$, then $a^{\bG}b^{\bG}$ 
 is contained in the closure of $(ab)^{\bG}$.
 In particular, the semisimple parts of elements of $a^{\bG}b^{\bG}$ 
 form a single semisimple conjugacy class of $\bG$.
\end{lem}

\begin{proof}
Let $\Gamma=\{(g,h) \in \bG \times \bG \mid gh^{-1} \in C_{\bG}(a)C_{\bG}(b)\}$.
Note that by assumption $\Gamma$ contains a dense open subset of
$\bG \times \bG$. Suppose that $(g,h)\in\Gamma$. Then
$$
(a^g, b^h) = (a^{gh^{-1}}, b)^h = (a^{xy}, b)^h = (a,b)^{yh},
$$
where $gh^{-1}=xy$ with $x \in C_\bG(a)$ and $y \in C_\bG(b)$.
Consider $f:\bG \times \bG \rightarrow \bG$ given by
$f(g,h)=a^gb^h$. If $c = ab$, then $f(\Gamma)\subseteq c^\bG$, whence
$f(\bG \times \bG)$ is contained in the closure of $c^\bG$, 
and the first part of the lemma follows.

Let $s$ be the semisimple part of $c$. Let $\bG_s$ be the set of elements  in 
$\bG$ whose semisimple part is conjugate to $s$.  As previously noted,
$\bG_s$ is a closed subvariety of $\bG$. Thus,  
$a^{\bG}b^{\bG} \subseteq \overline{c^\bG} \subseteq \bG_s$.
\end{proof}

We record the following trivial observation. Let $H$ and $K$ be subgroups of
a group $G$ and set $\Gamma: = G/H \times G/K$.  Then $G$ acts naturally on
$\Gamma$ and the orbits
of $G$ on $\Gamma$ are in bijection with the orbits of $H$ on $G/K$ and so
in bijection with $H\backslash G/ K$. In particular, this implies:

%%\cite[2.4]{FA}

\begin{lem}
 Let $G$ be a group with $a, b \in G$. The number of conjugacy classes
 in $a^Gb^G$ is at most $|C_G(a)\backslash G/C_G(b)|$.
\end{lem}

We record the following easy result.

\begin{lem}   \label{lem:fusion}
 Let $\bG$ be a connected reductive algebraic group with $\bH$ a connected
 reductive subgroup.  If $a, b \in \bH$ and the semisimple parts of
 $a^{\bH}b^{\bH}$ are not a single $\bH$-class, then the  semisimple parts
 of $a^{\bG}b^{\bG}$ are a union of an infinite number of $\bG$-conjugacy
 classes.
\end{lem}

\begin{proof}
Let $\bS$ be a maximal torus of $\bH$ and $\bT$ a maximal torus of $\bG$
containing $\bS$. By Lemma \ref{lem:irred},
$\overline{a^{\bH}b^{\bH}} \cap \bS$ is infinite.  In particular, 
$\overline{a^{\bG}b^{\bG}} \cap \bT$ is infinite and the result follows
by another application of Lemma \ref{lem:irred}.  
\end{proof}

We next point out the following short proof about products of centralizers.
For unipotent elements, this was proved independently by Liebeck and
Seitz \cite[Chapter 1]{LSbook}. We will obtain stronger results below.

\begin{cor}   \label{unipotent}
 Let $\bG$ be a semisimple algebraic group. If $a \in \bG$ is not central and
 $g \in \bG$, then $C_\bG(a)C_\bG(a^g)$ is not dense in $\bG$.
\end{cor}

\begin{proof}
Clearly, we can reduce to the case that $\bG$ is simple.
Write $a=su$ where $su=us$, $s$ is semisimple and $u$ is unipotent.
If $u \ne 1$, then $u$  is not central and since $C_\bG(u) \ge C_\bG(a)$,
we may assume $a=u$.  If $u=1$, then $a$ is semisimple.
In particular, we may assume that  $a$ is either semisimple or unipotent.

As we have noted in Lemma~\ref{lem:dense centralizer}, if $C_\bG(a)C_\bG(a^g)$
is dense in $\bG$, then the
semisimple parts of elements of $a^{\bG}a^{\bG}$ form a single conjugacy class.
If $a$ is semisimple, then we may assume that $a$ lies in a maximal torus $\bT$
and does not commute with some root subgroup $\bU_{\alpha}$. However,
by Lemma \ref{lem:d=2} applied to any large enough field $\FF_q$ that contains
an eigenvalue of $a$, $a^{\bH}a^{\bH}$ contains more than one semisimple class
in $\bH := \langle \bU_{\pm\alpha},\bT\rangle$, whence the result follows by
Lemma \ref{lem:fusion}. 

If $a \ne 1$ is unipotent, then by Lemma~\ref{lem:closures1}, there is a
positive root $\alpha$ and a nontrivial element $b\in\bU_{\alpha}$ in the
closure of $a^\bG$. Set $\bH := \langle \bU_{\alpha}, \bU_{-\alpha} \rangle$,
a rank~1 group. By a direct computation in $\SL_2$, we see that
$b^{\bH}b^{\bH}$ contains both non-central semisimple and unipotent elements.
The result now follows by Lemma \ref{lem:dense centralizer}.  
\end{proof}

Note that Corollary \ref{cor:prasad} now follows since
$C_{\bG}(a)g^{-1} C_{\bG}(a) = C_{\bG}(a)C_{\bG}(a^g)g^{-1}$.  

\begin{lem}   \label{lem:sp2 closures}
 Let $\bG=\Sp_{2n}(k) = \Sp(V)$ where $k$ is an algebraically closed field of
 characteristic $2$. Let $g \in \bG$ be a nontrivial unipotent element
 that is not a transvection. Let $h \in \bG$ be a unipotent element such that
 $V = V_1 \perp V_2 \perp V_3$ with $\dim V_1 = \dim V_2=2$ such that
 $h$ induces a transvection on $V_1$ and $V_2$ and is trivial on $V_3$.
 \begin{enumerate}[\rm(a)]
  \item Suppose that $(gv,v) =0$ for all $v \in V$.
   Then $g^2=1$, and the closure of $g^\bG$ contains short root elements
   but not long root elements.
  \item The closure of $h^\bG$ contains both short and long root elements.
  \item Suppose that $(gv , v) \ne 0$ for some $v \in V$.
   Then the closure of $g^\bG$ contains $h$ and so also both
   long and short root elements.
 \end{enumerate}
\end{lem}

\begin{proof}
In~(a) write $g = I + N$ where $N$ is nilpotent.
Note that $(Nv,v)= (gv -v, v)= 0$ for all $v \in V$.
Note also that $0=(N(v+w),v+w)= (Nv,w)+(Nw,v)$ and so
$(Nv,w)=(Nw,v)$ for all $v, w \in V$. It follows
that $NV \subseteq (\ker N)^{\perp}$, whence we see
that $N$ (and $g$) act trivially on a maximal totally singular
subspace $W$ of $V$. Let $\bP$ be the stabilizer of $W$ and $\bQ$ its unipotent
radical. We may view $\bQ$ as the space of symmetric $n \times n$ matrices.
Let $\bQ_0$ be the subspace of skew symmetric matrices.
Thus, $g \in \bQ$, whence $g^2=1$. Moreover the condition that $(gv,v)=0$
is exactly equivalent to $g \in \bQ_0$.

Since $(gv, v)=0$ for all $v \in V$ is a closed condition,
any element in the closure of $g^\bG$ also satisfies this,
whence long root elements are not in the closure of $g^\bG$
(and so necessarily short root elements are --- this is also
obvious from the proof above). This proves (a).

To prove (b) it suffices to work in $\Sp_4$. Note that we can conjugate
$h$ and assume that it is in the unipotent radical $\bQ$ of the stabilizer of
a maximal totally singular space. Note that $h^\bG \cap \bQ$ is dense in $\bQ$ and since
$\bQ$ contains both long and short root elements, the result follows.

Now assume that $(gv,v) \ne 0$ for some $v \in V$.
Recall that $g$ is not a transvection.

Choose $0 \ne w \in V$ with $gw=w$. Let $\bP$ be the subgroup of $\bG$ 
stabilizing the line containing $w$ and let $\bQ$ be its unipotent radical.
Note that $(gu, u) \ne 0$ for some $u$ with $(u,w) \ne 0$
(if $(gu,u)=0$ for all $u$ outside $w^{\perp}$, then
$(gu,u) = 0$ for all $u$ by density). Let $X = ku + kw$
which is a nondegenerate $2$-dimensional space and set $Y=X^{\perp}$.

Let $\bL$ be the Levi subgroup of $\bP$ that stabilizes $ku$ and $kw$ (and so
also $X$ and $Y$). Let $\bT$ be the $1$-dimensional central torus of $\bL$.

With respect to the decomposition
$V=kw \oplus Y \oplus ku$ , $g$ acts as
$$
\begin{pmatrix} 1 & s &c \\
                             0 & r &s^{\top} \\
                              0&0 &1 \\
\end{pmatrix}
$$
where $r \in \Sp(Y)$ is a unipotent element and $c \ne 0$.  First suppose
that $r$ is nontrivial. 
Thus, we see that the closure of $g^T$ contains an element of the same form
but with $s=0$. Since the closure of $r$ in $\Sp(Y)$ contains a root element,
we see that we may assume that $V = X \perp Y$, $g$ induces a transvection
on $X$ and $\dim Y =2$ or $4$ and $g$ induces either a transvection
on $Y$ or a short root element. If $g$ induces a transvection
on $Y$, then $g$ is conjugate to $h$ and there is nothing more
to prove. So assume that $\dim Y =4$ and $g$ acts as
a short root element on $Y$. This implies that the fixed space of $g$ is
a $3$-dimensional totally singular subspace $Z$. The hypotheses imply that
the closure of $g^{\bG}$ contains the unipotent radical of the stabilizer of
$Z$, whence it contains $h$.   Finally suppose that $r$ is trivial.   Since
$g$ is not a transvection, $s$ is nontrivial.  Since $\Sp$
is transitive on nonzero vectors, we can then assume that $s =(1,0,\ldots,0)$
and so reduce to the case of $\Sp_4$.  In that case, $g$ is already conjugate
to $h$.  This completes the proof. 
\end{proof}

\begin{lem}   \label{lem:closures2} 
 Let $\bG$ be a simple algebraic group over an algebraically closed field
 $k$ of characteristic $p \ge 0$. Let $g$ be a nontrivial unipotent element of
 $\bG$. The closure of $g^{\bG}$ contains long root elements unless one of the
 following occurs:
 \begin{enumerate}[\rm(a)]
   \item $(\bG, p) = (G_2,3)$ or $(F_4,2)$ and $g$ is a short root element; or
   \item $\bG=\Sp_{2n}=\Sp(V)$, $p=2$, $n\ge2$ and $(gv,v)=0$ for all $v\in V$.
 \end{enumerate}
 Moreover, if $(\bG,p) = (G_2,3)$ or $(F_4,2)$ and $g$ is not a root element,
 then the closure of $g^{\bG}$ contains both short and long root elements.
\end{lem}

\begin{proof}
By Lemma \ref{lem:closures1}, the result follows unless $\bG$ has two root
lengths. 

If $\bG=G_2$, see \cite [II.10.4]{Spa}. Similarly if $\bG=F_4$ with
$p=2$, see \cite [p. 250]{Spa}.

Now assume that $p \ne 2$ and $\bG=B_n$, $C_n$ or $F_4$.
It suffices to show that for $g$ a short root element, the
closure of $g^{\bG}$ contains long root elements. By passing to a rank $2$
subgroup containing both long and short root subgroups, it suffices to consider
$\bG=\Sp_4 = \Sp(V)$. In this case, we can write $V = V_1 \perp V_2$
where $g$ acts as a transvection on each $V_i$ and so clearly the closure of
$g^{\bG}$ contains long root elements (for $\bU$ a maximal unipotent subgroup
of $\Sp(V_1)\times \Sp(V_2)$, $g^{\bG} \cap \bU$ is dense in $\bU$ and
$\bU$ contains long root elements).

Finally, when $p=2$ and $\bG=\Sp_{2n}=\Sp(V)$ we may apply
Lemma \ref{lem:sp2 closures}. 
\end{proof}

\begin{lem}   \label{lem:so closures}
 Let $\bG=\SO_{2n+1}(k) = \SO(V)$, $n \ge 2$, with $k$ an algebraically closed
 field of characteristic $p \ne 2$. Let $g \in \bG$ be unipotent. Then the
 closure of $g^\bG$ contains a short root element if and only
 if $g$ has a Jordan block of size at least $3$.
\end{lem}

\begin{proof}
Clearly, the condition is necessary since having all Jordan blocks of size
at most $2$ is a closed condition and a short root element has a Jordan block
of size $3$.
Conversely, suppose that $g$ has a Jordan block of size $d \ge 3$.
It is well known that $V$ can be written as an orthogonal direct sum of
$g$-invariant subspaces on each of which either $g$ has a single Jordan block
of odd size or it has two Jordan blocks of (the same) even size of $g$.

Thus, we can write $V = V_1 \perp V_2$ where either $\dim V_1 = d \ge 3$
is odd and $g$ acting on $V_1$ is a regular unipotent element of $\SO(V_1)$
or $\dim V_1 = 2d \ge 6$ and $g$ acts on $V_1$ with two Jordan blocks of size
$d$.  By taking closures, we may assume that $g$ is trivial on $V_2$.
In the first case, the closure of $g^\bG$ contains all unipotent elements of
$\SO(V_1)$ (in particular a short root element). In the second case, we see
that $g$ is contained in some $\GL_d$ Levi subgroup of $\SO(V_1)$ and so $g$
is a regular unipotent element of $\GL_d$. Thus, its closure contains all
unipotent elements of $\GL_d$, whence in particular an element with two
Jordan blocks of size $3$. Now argue as in the first case.
\end{proof}

We next need a result about subgroups generated by root subgroups of a given
length.

\begin{lem}   \label{lem:ss}
 Let $\bG$ be a simply connected  algebraic group over an algebraically
 closed field $k$ of characteristic $p \ge 0$. Let $\bT$ be a maximal torus
 of $\bG$ and let $\Phi$ denote the set of roots of $\bG$ with respect to $\bT$.
 Assume that $\Phi$ contains roots of two distinct lengths. Let $\Phi_\ell$
 denote the long roots in $\Phi$ and $\Phi_s=\Phi\setminus\Phi_\ell$ the short
 roots. Let $X_{\ell}=\langle\bU_{\alpha}\mid \alpha \in\Phi_\ell\rangle$,
 and $\bX_s=\langle\bU_{\alpha}\mid\alpha \in\Phi_s\rangle$.
 The following hold:
 \begin{enumerate}[\rm(a)]
  \item $C_\bG(\bX_s)=Z(\bG)$.
  \item If $\bG=G_2$, then $C_\bG(\bX_{\ell})$ has order $3$ if $p \ne 3$ and is
   trivial otherwise.
  \item If $p=2$ and $\bG \ne G_2$,  then $C_\bG(\bX_{\ell})=Z(\bG)$.
  \item If $p \ne 2$ and $\bG \ne G_2$, then $C_\bG(\bX_{\ell})$ is an elementary
   abelian $2$-group and intersects a unique non-central conjugacy class of
   involutions unless $\bG=\Sp_{2n}$ in which case it intersects every
   conjugacy class of involutions (in $\Sp_{2n})$.
 \end{enumerate}
\end{lem}

\begin{proof}
This is a straightforward observation.  In fact if $p \ne 2$, then  $\bX_s = \bG$ unless $\bG=G_2$
with $p=3$.  In all those cases, the centralizer is just the center.
So we only need to consider $\bX_{\ell}$. If $\bG = F_4$, then 
$\bX_{\ell}\cong D_4$
while if $\bG=\Sp_{2n}$,  $\bX_{\ell}\cong \SL_2 \times \ldots \times \SL_2$.
Finally if $\bG = B_n$ with $p \ne 2$, then $\bX_{\ell} \cong D_n$. The result
follows.
\end{proof}

We can now prove Theorem \ref{main:alg} which we restate.
As we have already remarked, the result is essentially independent of the
isogeny type
of the simple algebraic group. We will work with the most convenient form
for each group (in particular, we work with $\Sp_{2n}$ and $\SO_{2n+1}$).

\begin{thm}   \label{thm:algebraic}
 Let $\bG$ be a simple algebraic group over an algebraically closed field $k$
 of characteristic $p \ge 0$. Let $a, b$ be non-central elements of $\bG$.
 Then one of the following holds (up to interchanging $a$ and $b$ and up to an 
 isogeny for $\bG$):
 \begin{enumerate}
  \item[\rm(1)] There are infinitely many semisimple conjugacy
   classes which occur as the semisimple part of elements of $a^{\bG}b^{\bG}$.
  \item[\rm(2)] $\bG= \Sp_{2n}(k) = \Sp(V)$, $n \ge 2$, $\pm b$ is a long root
   element, and either
   \begin{enumerate}
   \item[\rm(a)] $p \ne 2$ and $a$ is an involution; or
   \item[\rm(b)] $p=2$ and $a$ is an involution with $(av,v)=0$
    for all $v$ in $V$.
  \end{enumerate}
  \item[\rm(3)]  $\bG=\SO_{2n+1}(k)=\SO(V)$, $n\ge2$, $p\ne 2$ and
   $-a$ is a reflection and $b$ is a unipotent element with all Jordan blocks
   of size at most $2$.
  \item[\rm(4)]  $\bG=G_2$, $p \ne 3$, $a$ is of order $3$ with centralizer
   $\SL_3$ and $b$ is a long root element.
  \item[\rm(5)]  $\bG=F_4$, $p \ne 2$, $a$ is an involution with centralizer
   of type $B_4$ and $b$ is a long root element.
  \item[\rm(6)]  $(\bG,p)=(F_4,2)$ or $(G_2,3)$, $a$ is a long root element
   and $b$ is a short root element.
 \end{enumerate}
\end{thm}

\begin{proof}
Let $A$ be the closure of $a^\bG$ and $B$ the closure of $b^\bG$. Note
that if the closure of $a^{\bG}b^{\bG}$ contains only finitely many semisimple
classes,
the same is true for $AB$ (take closures). Thus, the same is true for $A'B'$
where $A' \subset A$ and $B'\subset B$ are conjugacy classes.

Also recall (Lemma \ref{lem:fusion})  that if $a, b \in \bH$ a connected 
reductive subgroup of $\bG$ and
there are infinitely many semisimple classes occurring as the semisimple
part of elements of $a^{\bH}b^{\bH}$, then the same is true in $\bG$.

\medskip
A) We give a very quick proof in the case that $\bG$ has only one root length
where we show that it is always the case that $a^{\bG}b^{\bG}$ contains 
infinitely many classes with distinct semisimple parts.

If the semisimple part $s$ of $a$ is noncentral, then $s$ is in the closure of
$a^{\bG}$ and so we may assume that $a$ is semisimple.  If not, then modifying
$a$ by a central element, we may assume that $a$ is unipotent.   
Similarly, we may assume that $b$ is either semisimple or unipotent. 

 If $a$ and $b$
are both semisimple, choose a maximal torus $\bT$ containing conjugates
$a', b'$ of $a$ and $b$. By conjugating by Weyl group elements, we may
assume that $a', b'$ do not commute with $\bU_{\alpha}$ for some root
$\alpha$. Thus, $\langle \bT, \bU_{\pm\alpha}\rangle $ is reductive
with semisimple part $A_1$. Moreover, $a', b'$ are not central, whence
the result follows from the result for $A_1$ (see Lemma \ref{lem:d=2}).
Similarly if $a$ and $b$ are both unipotent, then by replacing $a$ and
$b$ by elements in the closures of the classes, we may assume that
$a$ and $b$ are both long root elements, whence as above we reduce
to the case of $A_1$. If $a$ is unipotent and $b$ is semisimple, then
as above, we may assume that $a \in \bU_{\alpha}$ and $b \in \bT$
does not centralize $a$, whence again the result follows by the case
of $A_1$.

\medskip
B) So for the rest of the proof we assume that $\bG$ has two root lengths.
In particular $\rk(\bG) > 1$.
The proof is similar to that above but more complicated (and there
are always exceptions).

\vskip 0.5pc\noindent
Case 1. $a, b$ are both semisimple.

Let $\bT$ be a maximal torus containing both $a$ and $b$.
We apply Lemma \ref{lem:ss}.  In particular, we can choose a (short) root
subgroup
$\bU_{\alpha}$ and conjugates of $a, b$ by elements of the Weyl group  that
 do not centralize $\bU_{\alpha}$.  Now the result follows by considering
the subgroup $\langle \bT, \bU_{\pm\alpha}\rangle$.

\vskip 0.5pc\noindent
Case 2. $a$ and $b$ are both unipotent and are not among the excluded cases.

If the closures of $A$ and $B$ both contain long root elements, then
the result follows from the case of $A_1$. If $p \ne 2$, this is
always the case by Lemma \ref{lem:closures2} unless $(\bG,p)=(G_2,3)$.
If $\bG=G_2$ with $p=3$ or $\bG=F_4$ with $p=2$,
aside from the excluded cases, the closures of $A$ and $B$ will either
contain both long root elements or short root elements and again
the result follows.

It remains only to consider $\bG=\Sp_{2n}$ with $p=2$. It follows by
Lemma~\ref{lem:closures2} that unless $a$ or $b$ is a long root element, the
closures of $A$ and $B$ will contain short root elements and the result follows
as above. So we may assume that $b$ is a long root element and that the
closure of $A$ does not contain long root elements. Again by
Lemma~\ref{lem:closures2} this implies
that $a$ is an involution with $(av,v)=0$ for $v \in V$.

\vskip 0.5pc\noindent
Case 3. $a$ is semisimple and $b$ is unipotent.

Let $a \in \bT$ be a maximal torus.
If $\Char k = 2$ with $\bG \ne G_2$, we can choose a root subgroup
$\bU_{\alpha}$ with $a$ not centralizing $\bU_{\alpha}$ and reduce
to $\langle \bT, \bU_{\pm\alpha}\rangle$.
If $\bG=G_2$ with $p=3$, the same argument suffices.

Indeed, if the closure of $b^\bG$ contains a short root element, then
it suffices to assume that $b$ is contained in a short root subgroup
$\bU_{\alpha}$ and as above, we can conjugate $a$ by an element
of the Weyl group and assume that $a$ does not centralize $\bU_{\alpha}$.
Now argue as before.

The same argument suffices if $a$ is not an involution conjugate to an
element of the centralizer of the subgroup of $\bG$ generated by the long
root subgroups (with respect to $\bT$).   So we have reduced to the case
that $a$ is such an involution and the closure of $b^\bG$ contains long root
elements and not short root elements.  By Lemma \ref{lem:closures2}, these
are precisely the exceptions allowed in the theorem.

\vskip 0.5pc\noindent
Case 4. The general case.

We may assume (by interchanging $a$ and $b$ if necessary and using the
previous cases) that $a=su=us$ where $s$ is a noncentral semisimple element
and $u \ne 1$ is unipotent. 

If the semisimple part of $b$ is not central, we can take closures and so
assume that $b$ is semismple.  If the semisimple part of $b$ is central,
we can replace $b$ by a central element times $b$ and assume that $b$ is
unipotent. 

If $b$ is unipotent, then we can take closures and assume that $b$ is a root
element. By working in the closure of $a^G$
(which contains $s$), we see that by previous cases, it must be that 
$s^{\bG}b^{\bG}$ must have constant semisimple part. This implies
that either $p \ne 2$, $\bG \ne G_2$ and $s$ is an involution with $b$
a long root element or $\bG=G_2$, $p \ne 3$,  $s$ is an element of order $3$
and $b$ is a long root element.  

Let $\bT$ be a maximal torus. We may assume that $b \in \bU_{\alpha}$, a root
subgroup with respect to $\bT$. By taking closures in $D:=C_\bG(s)$, we may
also assume that $u$ is in a root subgroup $\bU_\beta$ with respect to $\bT$.
Thus, by considering $\langle \bT, \bU_{\pm \alpha}, \bU_{\pm \beta}\rangle$,
it suffices to assume that $\bG$ has rank $2$.   

Now suppose that $p \ne 2$ and  $\bG=\Sp_4$. As noted above,  $s$ must be
an involution. Note that $D$ contains both long and short root elements and
moreover the centralizer of $s$ is an $A_1A_1$, whence we see that there
are conjugates of $b$ and $a$ in $D$ with $a^Db^D$ having infinitely many
different semisimple parts.   

The remaining case is $p \ne 3$ and $\bG= G_2$. It follows that 
$s$ is an element of order $3$ with centralizer $D$ isomorphic to $A_2$.
So $u$ is a long root element. As we noted, $b$ is also a long root element
and so conjugate to an element of $D$. The result follows since it holds for
$A_2$. 
\end{proof}

We will discuss the examples listed above in the next section. In particular,
we will see that in all cases $a^{\bG}b^{\bG}$ is a finite union of classes but
always more than one. Indeed, we will see that $a^{\bG} \times b^{\bG}$ is the
union of a very small number of $\bG$-orbits (but always at least $2$). In
particular, this implies the following result which includes Szep's conjecture
for algebraic groups. See \cite{FA} for the finite case and \cite{Br1, Br2} for
related results on factorizations.

\begin{cor}   \label{cor:algAH}
 Let $\bG$ be a simple algebraic group. If $a, b$ are non-central elements
 of $\bG$, then $a^{\bG}b^{\bG}$ is not a single conjugacy class and
 $\bG \ne C_\bG(a)C_\bG(b)$.
\end{cor}

Another immediate consequence is:

\begin{cor}   \label{cor:factorize}
 Suppose that $\bG$ is a simple algebraic group over an algebraically closed
 field $k$ with $\Char k =p \ge 0$, and that $a, b$ are non-central
 elements of $\bG$.
 If $C_\bG(a)C_\bG(b)$ is dense, then $\bG, a, b$ are as described in
 Theorem \ref{thm:algebraic}. In particular, $C_\bG(a)C_\bG(b)$ is not dense
 if any of the following hold (modulo the center):
 \begin{enumerate}[\rm(i)]
  \item $a$ and $b$ are conjugate;
  \item neither $a$ nor $b$ is unipotent; or
  \item $a$ is semisimple and has order greater than $3$.
 \end{enumerate}
\end{cor}

Indeed, if $a$ is semisimple and is not an involution then
$\bG = G_2$ and $a$ has order $3$.

We point out one further corollary which also comes from analyzing the
exceptions in the theorem above.

\begin{cor}  \label{cor: borel vs. finite}
 Let $\bG$ be a semisimple algebraic group. Let $a, b \in \bG$.
 The following are equivalent.
 \begin{enumerate}[\rm(i)]
  \item $a^{\bG}b^{\bG}$ is a finite union of conjugacy classes.
  \item The closure of $a^{\bG}b^{\bG}$ contains only one semisimple conjugacy
   class.
  \item $\langle a, b^g \rangle$ is contained in some Borel subgroup of $\bG$
   for every $g \in \bG$.
  \item $|C_\bG(a)\backslash \bG /C_\bG(b)|$ is finite.
    \item $C_\bG(a)C_\bG(b^g)$ is dense in $\bG$ for some $g \in \bG$.
  \item $\bG$ has finitely many orbits on $a^\bG \times b^\bG$.
 \end{enumerate}
\end{cor}

In fact, we will see in the next section 
that in all the cases where $a^{\bG}b^{\bG}$ is a finite
union of conjugacy classes, it is a union of at most $4$ classes.
%% probably 3
%%%%%%%%%%%%%%%%%%%%%%%%%%%%%%%%%%%%%%%%%%%%

\section{Examples with dense centralizer products}   \label{sec:dense}

We now consider the examples for the exceptions in Theorem~\ref{main:alg}
and show that $a^{\bG}b^{\bG}$ is a finite union of conjugacy classes in all
cases. However, it always consists of at least two classes and so
$\bG \ne C_\bG(a)C_\bG(b)$.
We also show that there is a dense (and so open) element 
in $C_\bG(a)\backslash \bG /C_\bG(b)$, 
whence $C_\bG(a)C_\bG(b^g)$ is dense for some $g \in \bG$.  Indeed we will
see that $|C_\bG(a)\backslash \bG /C_\bG(b)| \le 4$ in all cases. 

Throughout the section, fix $k$ to be an algebraically closed field
of characteristic $p \ge 0$.

\begin{exmp}   \label{ex: G2}
Let $\bG=G_2$ with $p=3$. Let $a$ be a long root element and $b$ a short
root element. Choose conjugates so that $ab$ is a regular unipotent
element. Then $\dim a^\bG=\dim b^\bG=6$ and $\dim (ab)^\bG = 12$.
Since $\dim a^\bG + \dim b^\bG = 12$, we see that $\dim a^{\bG}b^{\bG} \le 12$
and so $(ab)^\bG$ is the dense orbit in $a^{\bG}b^{\bG}$.
In particular, $\overline{a^{\bG}b^{\bG}}$ is the set of unipotent elements in 
$\bG$. Moreover, for such a pair $(a,b)$ we see that
$\dim C_\bG(a) \cap C_\bG(b) \ge 2$ because $\dim C_\bG(a)=\dim C_\bG(b)=8$. 
However, since $C_\bG(a)\cap C_\bG(b)\le C_\bG(ab)$ and $\dim C_\bG(ab)=2$,
we have equality, whence $C_\bG(a)C_\bG(b)$ is dense in $\bG$. Note that there
are at least two classes in $a^{\bG}b^{\bG}$. As noted, $a^{\bG}b^{\bG}$
contains the regular unipotent
elements. On the other hand, we can find conjugates which commute and so the
product will have order~$3$ and so is not a regular unipotent element
(and so $\bG \ne C_\bG(a)C_\bG(b)$). Since $(a,b)^\bG$ is dense in 
$a^\bG \times b^\bG$, it follows that any pair in $a^\bG \times b^\bG$ is
contained in some Borel subgroup.

We next show that in fact $|C_\bG(a) \backslash \bG /C_\bG(b)|=2$. 
This can be seen as follows.  Fix a maximal torus $\bT$ and a Borel
subgroup $\bB$ containing $\bT$. Note that we may take $C_\bG(a) = \bP_1'$
and $C_\bG(b)=\bP_2'$ where $\bP_1$ and $\bP_2$ are the maximal parabolics
containing $\bB$. Let $\bT_i = \bT \cap \bP_i'$. Since the $\bP_1, \bP_2$
double cosets are in bijection
with the corresponding double cosets in the Weyl group, we see that there are
$2$ such double cosets. Note that if $w$ is in the Weyl group, then
$\bP_1' w \bP_2' = \bP_1' \bT_1(\bT_2)^ww\bP_2'$. Note that
$\bT=\bT_1(\bT_2)^w$ for any $w$ in the Weyl group. It follows that
$\bP_1'w\bP_2'=\bP_1'\bT w\bT\bP_2'=\bP_1w\bP_2$ and so  
$|\bP_1'\backslash \bG / \bP_2'|=2$.
\end{exmp}

\begin{exmp}   \label{ex:G2ss}
Let $\bG=G_2$ with $p \ne 3$. Let $a$ be a long root element
and $b$ an element of order $3$ with centralizer $\SL_3$.
First take $k$ to be the algebraic closure of a finite field.
By Lemma~\ref{lem:g2not3}, we see that $\bG$ has two orbits on 
$a^\bG \times b^\bG$
and for any $(c,d) \in a^\bG \times b^\bG$, $\langle c, d \rangle$ is contained
in some Borel subgroup.   As noted in the proof of Lemma \ref{lem:g2not3}, 
$a^{\bG}b^{\bG}$ consists of two conjugacy classes (the classes have
representatives  $bx$ and $by$
where $x$ is a long root element $C_{\bG}(b)$ and $y$ is regular unipotent
element of $C_{\bG}(b)$). 

By taking ultraproducts, we see that the same is true for some algebraically
closed field of characteristic $0$. By a well known argument
(cf. \cite[1.1]{GLMS}), it follows that the same is true for any algebraically
closed field of characteristic not~$3$.
\end{exmp}

\begin{exmp}   \label{ex:f4}
Let $\bG=F_4$ with $p=2$. Let $a$ be a long root element and $b$ be a short
root element. We will show that $\langle a, b \rangle$ is always unipotent.
Let $\bT$ be a maximal torus of $\bG$ with $\bT \le \bB$, a Borel subgroup
of $\bG$.
Let $\bP_1$ and $\bP_4$ be the two end node maximal parabolic subgroups
containing $\bB$. Note that there are only finitely many
$\bP_1, \bP_4$ double cosets  in $\bG$ each of the form $\bP_1w\bP_4$ where $w$
is in the Weyl group. We may assume that $\bP_1' = C_\bG(a)$ and $\bP_4'=C_\bG(b)$.
Arguing precisely as for $G_2$ with $p=3$, we see that $\bP_1'w\bP_4'=\bP_1w\bP_4$.
Thus there are only finitely many $C_\bG(a), C_\bG(b)$ double cosets in $\bG$. 
In fact, by computing in the Weyl group, we see that there are precisely
$2$ double cosets.
In particular, $\bG$ has only two orbits on $a^\bG \times b^\bG$, whence
the semisimple part of any element in $a^{\bG}b^{\bG}$ is the same up to conjugacy
and so is contained in the set of unipotent elements. The dense double
coset corresponds to $w$ being the element which acts as inversion on $\bT$.
One computes that the group generated by $a$ and $b^w$ is unipotent, whence
this is true for all pairs in $a^\bG \times b^\bG$ (by density).  

If $a$ and $b^g$ commute, then $ab^g$ has order $2$ while if $a$ and $b^g$
do not commute, we see that $ab^g$ has order $4$ (already in $C_2$).  Thus 
$a^{\bG}b^{\bG}$ consists of two conjugacy classes.  
\end{exmp}

\begin{exmp}   \label{ex:f4odd}
Let $\bG=F_4$ with $p \ne 2$. Let $a$ be a long root element and let $b$
be an involution with centralizer of type $B_4$. If $k$ is
the algebraic closure of a finite field of odd characteristic,
it follows by Lemma \ref{lem:f43a} that 
 $|C_\bG(a) \backslash \bG /C_\bG(b)|=2$
and every pair $(c,d) \in a^\bG \times b^\bG$ has the property
that $\langle c, d \rangle$ is contained in a Borel subgroup.
Arguing as for $G_2$ with $p \ne 3$, the same is true for
$k$ any algebraically closed field of characteristic not $2$.
Thus, $\bG$ has two orbits on $a^\bG \times b^\bG$. Clearly, one orbit
is the set of commuting pairs. If $a$ and $b^g$ commute,
then $(ab^g)^2$ is a long root element.  It is straightforward to compute
that if $a$ and $b^g$ do not commute, then $(ab^g)^2$ is a short root element
and so  there are exactly two conjugacy classes in $a^{\bG}b^{\bG}$.
\end{exmp}

\begin{exmp}   \label{ex:sp odd)}
Let $\bG=\Sp_{2n}=\Sp(V)$, $n \ge 2$ with $p \ne 2$.
Let $a$ be a transvection and let $b$ be an involution (i.e., all eigenvalues
are $\pm 1$). We claim that $\langle a, b \rangle$ is contained in a Borel
subgroup, whence $a^{\bG}b^{\bG}$ contains only elements with the semisimple 
part conjugate to $b$.   
Let $W$ be the intersection of the fixed spaces of $a$ and $b$.
If $W$ contains a nondegenerate subspace, we pass to the orthogonal complement
and use induction. If $W$ is totally singular, then $\dim W = n-1$ or $n$.
Let $\bP$ be the stabilizer of $W$ with unipotent radical $\bQ$. If $\dim W = n$,
$a$ is in $\bQ$, whence the result.
If $\dim W = n-1$, then $b$ is central in $\bP/\bQ$, whence the result follows
in this case as well.  

Since $C_\bG(b) = \Sp_{2m}\times\Sp_{2n-2m}$, we see that $C_\bG(b)$ has three
orbits on $V \setminus\{0\}$ whence $|C_\bG(a) \backslash \bG /C_\bG(b)|=3$. 
It is straightforward to see (already in $\Sp_4$) that $a^{\bG}b^{\bG}$ contains
elements whose square are long root elements or short root elements, whence
$a^{\bG}b^{\bG}$ contains at least $2$ conjugacy classes. 
\end{exmp}

\begin{exmp}   \label{ex: sp uni}
Let $\bG=\Sp_{2n}=\Sp(V)$, $n \ge 2$ with $p=2$. Let $a$ be a transvection
and $b$ an involution with $(bv,v)=0$ for all $v \in V$.
We claim that $\langle a, b \rangle$ is unipotent.
Consider the intersection $W$ of the fixed space of $b$ and the fixed space
of $a$. This has dimension at least $n-1 \ge 1$. If this space contains a
nondegenerate space $D$, we can replace $V$ by $D^{\perp}$ and use induction
(note that if $n=1$, $b=1$). So we may assume that $W$ is totally singular.
Let $\bP$ be the stabilizer of $W$ and $\bQ$ the unipotent radical of $\bP$.
If $\dim W = n$, then $a, b$ are both in $\bQ$
and so commute. If $\dim W={n-1}$, then $\langle a, b \rangle \le \Sp_2(k)\bQ$
whence $b \in \bQ$. Thus $a^{\bG}b^{\bG}$ is contained in the set of unipotent
elements (and any pair in $a^\bG \times b^\bG$ is contained in a common Borel
subgroup).  As we have seen $a$ and $b$ may commute and so $ab$ is an
involution but it is straightforward to see that the order of $ab$ may be~$4$.

We can write $V = V_1 \perp V_2 \perp \ldots \perp V_m \perp W$ where
$\dim V_i=4$,
and $b$ acts as a short root element on $V_i$ and
$b$ is trivial on $W$. If $n=2$, we argue as for $G_2$ to see that
$C_\bG(a)C_\bG(b)$ can be dense. Indeed, it follows that in general $C_{\bG}(b)$
has only finitely many orbits on $V$, whence there are only finitely many
$C_\bG(a), C_\bG(b)$ double cosets in $\bG$.   Indeed, it is a fairly easy
exercise in linear algebra to show that $C_\bG(b)$ has at most $4$ orbits
on nonzero vectors in $V$, whence  $|C_\bG(a) \backslash \bG /C_\bG(b)| \le 4$.
\end{exmp}

\begin{exmp}   \label{ex:so}
Let $\bG= \SO_{2n+1}=\SO(V)$, $n \ge 2$ with $p \ne 2$. Let 
%$-a$ be a reflection in $\bG$
$a\in\bG$ be such that $-a$ is a reflection,
and let $b$ be a unipotent element with all Jordan blocks of size at
most $2$. We claim that $\langle a, b \rangle$ is contained in a Borel
subgroup of $\bG$, whence $a^{\bG}b^{\bG}$ consists of unipotent elements.
If $n=2$, then the result follows by the result for $\Sp_4$.
So assume that $n > 2$.
Let $W$ be the intersection of the $-1$ eigenspace of $a$ and
$[b,V]$. Note that $W \ne 0$ (since $\dim [b,V] \ge 2$) and is totally
singular. By induction, $ab$ has semisimple part the negative of a reflection
on $W^{\perp}/W$, whence also in $\bG$.

Note that $C_\bG(a)$ is the stabilizer of a nonsingular $1$-space. Note also
that the number of Jordan blocks of $b$ is even, whence by reducing to the
$4$-dimensional case we see that $C_\bG(b)$ has only finitely many orbits
on $1$-dimensional spaces. Thus there are only finitely many
$C_\bG(a) \backslash \bG /C_\bG(b)$ double cosets. Indeed, it is a
straightforward exercise to see that $|C_\bG(a)\backslash\bG/C_\bG(b)|\le 4$.
By reducing to the
case of $\SO_5 \cong C_2$, we see that the unipotent parts of elements in
$a^{\bG}b^{\bG}$ are in at least $2$ different conjugacy classes, whence
$a^{\bG}b^{\bG}$ is not a single conjugacy class. 
\end{exmp}

%%%%%%%%%%%%%%%%%%%%%%%%%%%%%%%%%%%%%%%%%%%%%%%%%%%%%%%%%%%%%%%%%%%%%%%%%
\section{A short proof of Szep's conjecture}   \label{sec:szep}

We use our previous results to give a short proof of the conjecture of Szep's
(a finite simple group cannot be the product of two centralizers);
see \cite{FA} for the original proof. In particular, using
Corollaries~\ref{cor:ss} and~\ref{cor:factorize} (for semisimple elements),
we can shorten the proof considerably.

Observe the following connection to the Arad--Herzog conjecture.
Let $G$ be a group with $a, b \in g$.
As we have noted the number of orbits of $G$  on $a^G \times b^G$
is precisely $C_G(a)\backslash G /C_G(b)$. In particular,
if  $G=C_G(a)C_G(b)$, then
$G$ acts transitively on $a^G \times b^G$  and $a^Gb^G$ is a
single conjugacy class of $G$.  Indeed if $w(x,y)$ is any element of the
free group on two generators, then $w(a',b')$ is conjugate to $w(a,b)$
for all $(a',b') \in a^G \times b^G$. In particular, if the Arad--Herzog
conjecture holds for $G$, then no such factorization can exist.

\begin{thm}[Szep's conjecture; Fisman--Arad \cite{FA}]   \label{thm:szep}
 Let $G$ be a finite non-abelian simple group. If $a,b$ are non-trivial
 elements of $G$, then $G \ne C_G(a)C_G(b)$.
\end{thm}

\begin{proof}
For alternating groups, the Arad--Herzog conjecture, proved in
Theorem~\ref{thm:alt}, gives the result. For the twenty six sporadic groups,
it is straightforward to check the Arad--Herzog conjecture from the character
tables.
\par
So now assume that $G$ is simple of Lie type. The basic idea is as follows.
We find two primes $r_1,r_2$ for which the Sylow $r_i$-subgroups of $G$ are
cyclic and there exist regular semisimple elements $x_1,x_2\in G$ of order
$r_i$ such that no nontrivial element of $G$ centralizes conjugates of both
of them. \par
Then assume that $G =C_G(a)C_G(b)$ for $a,b\in G$. If $r_i$ divides
$|C_G(a)|$, then some conjugate of $x_i$ centralizes $a$, and similarly for
$|C_G(b)|$. Thus by our assumption, $a$ centralizes a conjugate of $x_1$,
say, and $b$ centralizes a conjugate of $x_2$. Since the $x_i$ are regular,
this implies that $a,b$ are both semisimple. But then by Proposition 
\ref{prop:regss},
$a^Gb^G$ consists of more than one class of $G$. As pointed out above this
implies that $C_G(a)C_G(b)\ne G$, a contradiction. \par
For $G$ of exceptional type and rank at least~4, we take for $r_1,r_2$
Zsigmondy primes as
%for the largest cyclotomic factors in the orders of the tori $T_i$
listed in Table~\ref{tab:exctori}. For the small rank cases the claim
follows from Proposition~\ref{prop:smallexc}.
\par
For $G$ of classical type, the claim for $\PSL_n(q)$ follows by
Theorem~\ref{thm:psl}, and for $\PSU_n(q)$ with $3\le n\le6$ by
Proposition~\ref{prop:smallPSU}. For the types $\tw2A_n$, $B_n$, $C_n$,
$\tw2D_n$ and $D_{2n+1}$, we take the two tori $T_1,T_2$ given in
\cite[Table 2.1]{MSW}, which contain Zsigmondy prime elements and are not
contained in a common centralizer (by the arguments given in \cite[\S2]{MSW}).
\par
This leaves only the case of $\OO_{4n}^+(q)$.
If $n=2$, we apply Proposition \ref{prop:smallclassical}. So assume that $n>2$.
Here we take $r_1$ to be a Zsigmondy prime divisor of $q^{4n-2}-1$,
$r_2$ to be a Zsigmondy prime divisor of $q^{2n-1}-1$.
Let $x_i \in G$ be of order $r_i$.
Note that the Sylow $r_i$-subgroups of $G$ are cyclic, and $x_1$ and
$x_2$ are regular semisimple. Abusing the notation, we will
let $x_i$ denote the inverse image of $x_i$ of order $r_i$ in
$S := \SO^+_{4n}(q)$. Then $C_S(x_1) \cong C_{q^{2n-1}+1} \times C_{q+1}$ and
$C_S(x_2) \cong C_{q^{2n-1}-1} \times C_{q-1}$. Suppose $s \in S$ centralizes
conjugates of both $x_1$ and $x_2$. Then $|s|$ divides
$\gcd(q^{2n+1}+1,q^{2n-1}-1)\le2$. In particular, $s = 1$ if $2|q$.
Assume $q$ is odd and $s \neq 1$. Since $s$ centralizes a conjugate of $x_1$,
we see that $s$ acts as $\pm 1$ on $U_1$ and as $\pm 1$ on $U_1^{\perp}$, where
$U_1$ is a nondegenerate subspace (of the natural $\FF_qS$-module
$V = \FF^{4n}_q$) of type $-$ of codimension $2$.
Similarly, since $s$ centralizes a conjugate of
$x_2$, $s$ acts as $\pm 1$ on $U_2$ and as $\pm 1$ on $U_2^{\perp}$, where
$U_2$ is a nondegenerate subspace of $V$ of type $+$ of codimension $2$.
This can happen only when $s = -1_V$. We have shown
that no nontrivial element of $G$ can centralize conjugates of both $x_1$ and
$x_2$, and so we can finish as above.
\end{proof}

We next give some examples to show that if the ambient group is not simple,
there are many counterexamples to both Szep's conjecture and the 
Arad--Herzog conjecture.
Of course, a trivial example is to take $G$
a direct product and choose elements in different factors. There is a more
interesting example for almost simple groups.

\begin{exmp}   \label{ex:as}
Let $G := \GL_{2n}(q)= \GL(V)$, $n \ge 1$, $q > 2$, $(n,q) \neq (1,3)$, 
and let $\tau$ be a graph automorphism of $G$ with centralizer 
$C_G(\tau) \cong \Sp_{2n}(q)$. Also, 
let $x = \diag(a,1, \ldots ,1)$ for some $1 \neq a \in \FF^{\times}_q$, so
that $C_G(x)$ is the stabilizer of a pair $(L, H)$,
where $L$ is a line and $H$ is a hyperplane not containing $L$ in $V$. 

First we show that $G=C_G(\tau)C_G(x)$; equivalently, 
$C_G(\tau)$ is transitive on such pairs $(L,H)$.
Since $\Sp_{2n}(q)$ is transitive on nonzero vectors, we just have to show
that the stabilizer of $L$ in $C_G(\tau)$ is transitive on the hyperplanes
complementary to $L$.
Let $H_i$, $i=1,2$, be fixed hyperplanes complementary to $L$.
Let $0 \ne v \in L$. Choose $v_i \in L_i:=H_i^{\perp}$ with $(v_i,v)=1$.
Set $M_i = \langle L, L_i \rangle$. Note that $V = M_i \perp H_i'$ where
$H_i'=L_i^{\perp}\cap H_i$ is a hyperplane in $L^{\perp}$ not containing $L$.
By Witt's theorem for alternating forms, there is an isometry $g \in G$ such
that $gM_1=M_2$ and $gH_1=H_2$. So we may assume that $M_1=M_2$. Applying
another isometry, we may assume that $L_1=L_2$ whence $H_1=H_2$ as required.

It follows that $\tau^A x^A = (\tau x)^A$
with $A := \langle G, \tau \rangle$. The same also holds in the 
almost simple group $A/Z(G) \leq \Aut(\PSL_n(q))$.
\end{exmp}

Of course this also works for the algebraic group (or indeed over any field of
size greater than $2$).

Here is another example.

\begin{exmp}   \label{ex:wreath}
Let $L$ be a nontrivial finite group and $H$ a cyclic group of order $n > 1$.
Set $G = L \wr H$. Let $1 \ne a$ be an element of $L^n$ with only one
nontrivial coordinate. Let $b$ be a generator for $H$. Note that
$C_G(a) \ge L^{n-1}$ while $C_G(b) = D \times H$ where $D$ is a diagonal
subgroup of $L^n$. Thus $G=C_G(a)C_G(b)$ and $a^Gb^G=(ab)^G$.
\end{exmp}

In particular, we can take $n=2$, $L$ simple non-abelian and choose
$a$ and $b$ to be involutions or $n=p$ a prime and $L$ simple of order
divisible by $p$ and choose $a$ and $b$ to have order $p$. Note
that since $G$ is transitive on $a^G \times b^G$, we see that
$\langle a^x, b^y \rangle$ is always a $p$-group.

We give one more example to show that $a^Gb^G=(ab)^G$ does not
necessarily imply that $G=C_G(a)C_G(b)$.

\begin{exmp}   \label{ex:normal}
Let $G$ be a group with a normal subgroup $N$. Suppose that
$a, b \in G$ are such that all elements in $abN$ are conjugate.
Assume that $G/N$ is abelian.
Then clearly, $a^Gb^G=(ab)^G = abN$ (the condition that $G/N$
is abelian can be relaxed). Such examples include $\fA_4$ and
non-abelian groups of order $qp$ where $p < q$ are odd primes
with $a, b$ classes of $p$-elements with $b$ not conjugate to $a^{-1}$.
\end{exmp}

%%%%%%%%%%%%%%%%%%%%%%%%%%%%%%%%%%%%%%%%%%%%%%%%%%%%%%%%%%%%%%%%%%%%%%%%%
\section{Variations on Baer--Suzuki}  \label{sec: baersuz}
Recall that the Baer--Suzuki theorem asserts that if $G$ is a finite (or
linear) group, $x \in G$, then $\langle x^G \rangle$ is nilpotent if and only
if $\langle x, x^g \rangle$ is nilpotent for all $g\in G$. One might ask
what happens if we assume that $\langle x, y^g \rangle$ is nilpotent
(or solvable) for all $g \in G$ for $x, y$ not necessarily conjugate elements.
The examples in Section~\ref{sec:szep} show that this analog of the
Baer--Suzuki theorem fails for nonconjugate elements (and indeed even the
solvable version of Baer--Suzuki fails --- see \cite{guest}). As we have seen
for $p=2,3$, we even have counterexamples for simple algebraic groups (and
so also for finite simple groups).

However, it turns out that one can extend the Baer--Suzuki theorem with
appropriate hypotheses at least for $p$-elements with $p \ge 5$ (see
Theorem~\ref{thm:bsgen} below). 

\subsection{Some variations on Baer--Suzuki for simple groups}
First we note that in Theorem \ref{bs1} there are no exceptions if $p>3$. 
Moreover, the same proof
(basically reducing to the case of rank~$1$ groups) gives the following:

\begin{cor}   \label{bs5}
 Let $G$ be a finite simple group of Lie type in characteristic $p\ge 5$.
 Let $u, w$ be nontrivial unipotent elements of $G$. There exists $g \in G$
 such that $uw^g$ is not unipotent and $\langle u,w^g\rangle$ is not solvable.
\end{cor}

Guest \cite{guest} proved that if $G$ is a finite group with $F(G)=1$
and $x \in G$ has prime order $p \ge 5$, then
$\langle x, x^g \rangle$ is not solvable for some $g \in G$.   
See also \cite{GGK}.   

Next we record the following results for alternating and sporadic groups.

\begin{lem}   \label{lem:alt}
 Let $G=\fA_n$, $n\ge5$. Let $p$ be a prime with $p\ge3$. If $u,w\in G$
 are nontrivial $p$-elements, then there exists $g \in G$ such that $uw^g$
 is not a $p$-element and $\langle u, w^g \rangle$ is nonsolvable.
\end{lem}

\begin{proof}
First take $p=3$. By induction, it suffices to consider the case $n=5$ or~$6$
where the result is clear. So assume that $p \ge 5$. Clearly, it suffices
to consider the case $p=n$, where again the result is clear.
\end{proof}

\begin{lem}   \label{lem:sporadic}
 Let $G$ be a sporadic simple group. Let $p$ be a prime. Let $u, w \in G$ be
 nontrivial $p$-elements. Then
 \begin{enumerate}[\rm(a)]
  \item there exists $g \in G$ such that $uw^g$ is not a $p$-element; and
  \item if $p \ge 5$, there exists $g \in G$ such that
   $\langle u, w^g \rangle $ is not solvable.
 \end{enumerate}
\end{lem}

\begin{proof}
These are straightforward computations using {\sf GAP}.
\end{proof}

\subsection{A variation on Baer--Suzuki for almost simple groups}
Our next goal is to prove the following:

\begin{thm}   \label{thm:bsas}
 Let $p \geq 5$ be a prime and let $S$ be a finite non-abelian simple group.
 Let $S \lhd G \leq \Aut(S)$, and $c,d \in G$ any two elements of order $p$.
 Then:
 \begin{enumerate}[\rm(1)]
  \item There is some $g \in G$ such that $\langle c,d^g \rangle$ is
   not solvable.
  \item There is some $g \in G$ such that $cd^g$ is not a $p$-element.
 \end{enumerate}
\end{thm}

If $S$ is an alternating group or a sporadic group, we apply
Lemmas~\ref{lem:alt} and~\ref{lem:sporadic}. So assume that $S$
is of Lie type in characteristic $r$. In what follows, we will call any element
of $G$ inducing a nontrivial field automorphism of $S$ modulo $\Inndiag(S)$,
the subgroup of inner-diagonal automorphisms of $S$,
\emph{a field automorphism}. Also, $\Phi_m(t)$ denotes the
$m^{\mathrm {th}}$-cyclotomic polynomial in the variable $t$.

\subsubsection{The case $r=p$}

If $c,d$ are both inner elements, then the result follows by
Corollary~\ref{bs5}. So assume that $c$ induces a field automorphism
of $S$. Suppose that $S$ has rank at least $2$.
Let $P$ be a maximal end node parabolic subgroup of $S$
with $d$ not in the radical $Q$ of $P$. Note that $N_G(P)$ contains
a Sylow $p$-subgroup of $G$ and so we may assume that $c, d \in N_G(P)$.
Since $p\ge 5$, it follows that $P/Q$ has a unique simple section $S_0$ and
that $c,d$ each act nontrivially on $S_0$, whence the result follows by
induction.

If $S$ has rank $1$, then either $S \cong \PSL_2(q)$ or $\PSU_3(q)$.
Write $q=q_0^p$. Since $p \ge 5$, it follows \cite[7.2]{GL} that there
is a unique conjugacy class of subgroups of field automorphisms
of order $p$ and that every unipotent element
is conjugate to an element of the group defined over $\FF_{q_0}$. Thus,
we see conjugates of $c, d$ in $H:=\PSL_2(q_0) \times \langle c \rangle$ or
$\PSU_3(q_0) \times \langle c \rangle$.
Note that $c$ is conjugate in $G$ to a non-central element in $H$ (again by
\cite[7.2]{GL}) and so the result follows by induction.

\medskip
For the rest of the section, we assume that $r \ne p$.

\subsubsection{Field automorphisms}
Here we handle the case when $c$ is a field automorphism of order $p$ of $S$.
So we can view $S= S(q)$ as a group over the field of $q$ elements with
$q = q_0^p$. One can find a simple algebraic group $\bG$ of adjoint type
over $\overline{\FF}_r$ and a Steinberg endomorphism $F:\bG \to \bG$
such that $X = X(q) := \bG^{F^p}$ is the group of inner-diagonal automorphisms
of $S$. By \cite[7.2]{GL}, any two subgroups of $G$ of order $p$
of field automorphisms of $S$ are conjugate via an element of $X(q)$.
In particular, this implies
that any field automorphism normalizes a parabolic subgroup of any given type.
Thus, precisely as in the case $r=p$, if $d$ is also a field automorphism,
we can reduce to the case that $S$ has rank $1$ and complete the proof.

Thus, we may assume that $d$ is semisimple. Moreover, since $d^S =d^X$, it
suffices to work with $X$-classes and as noted there is a unique conjugacy
class of subgroups of order $p$ consisting of field automorphisms.
We digress to mention two results about $p$-elements.

\begin{lem}   \label{lem:rank}
 Let $\bH$ be a connected reductive algebraic group over $\overline{\FF}_r$,
 with a Steinberg endomorphism $F:\bH \to \bH$, and let $p\ne r$
 be a prime not dividing the order of the Weyl group $W$ of $\bH$ nor the
 order of the automorphism of $W$ induced by $F$. Then the Sylow
 $p$-subgroups of $\bH^F$ and $\bH^{F^p}$ are abelian of the same rank.
\end{lem}

\begin{proof}
Under our assumptions, by \cite[Thm.~25.14]{MT} the Sylow $p$-subgroups of
$\bH^{F^i}$ are homocyclic abelian, of rank $s_i$ say. Moreover, there is at
most one cyclotomic polynomial $\Phi_{e_i}$ dividing the order polynomial of
$(\bH,F^i)$ such that $p|\Phi_{e_i}(q^i)$, where $q$ denotes the absolute
value of the eigenvalues of $F$ on the character group of an $F$-stable
maximal torus of $\bH$, and $s_i$ equals the $\Phi_{e_i}$-valuation of
the order polynomial. Now $p|\Phi_e(q)$ if and only if $p|\Phi_{ep}(q)$, and
if $\Phi_e$ divides the order polynomial of $(\bH,F)$ then $\Phi_{ep}$
divides the one of $(\bH,F^p)$, to the same power. Thus, $e_p=pe_1$ and
$s_p=s_1$, and the claim follows.
\end{proof}

Note that in our situation the previous result says that if $p \geq 5$ does
not divide
the order of the Weyl group $W$, then every element in $S$ of order $p$ is
conjugate to an element centralized by $F$.  We can extend this even
%to the case when the Sylow $p$-subgroups are non-abelian and
to some primes dividing $|W|$.
%Since the cases of Ree and Suzuki groups are covered by the previous result,
%we will exclude those in the next result.

\begin{lem}   \label{lem:fieldofdef}
 Let $\bH$ be a simple simply connected linear algebraic group over
 $\overline{\FF}_r$, with a Steinberg endomorphism $F:\bH \to \bH$ and
 $5 \le p\ne r$ a prime.  
% Suppose $\bH^{F^p}$ is neither a Ree nor a Suzuki group. 
 If $x \in \bH^{F^p}$ has order $p$, then $x$ is conjugate in
 $\bH^{F^p}$ to an element of $\bH^F$.
\end{lem}

\begin{proof}
Since $\bH$ is simply connected, centralizers of semisimple elements
in $\bH$ are connected. So it suffices to show that the $\bH$-conjugacy class 
$C$ of $x$ is $F$-stable (see \cite[Thm.~26.7]{MT}).  
Since $F^p(x)=x$,  $C$ is $F^p$ invariant.  Thus, it suffices to
show that $F^m$ fixes $C$ for some $m$ prime to $p$.
Write $F^2=F_0\tau = \tau F_0$ where $F_0$ is
a standard Frobenius map with respect to an $\FF_q$-structure and $\tau$ is
a graph automorphism of $\bH$ of order $e \le 3$.  So we may replace
$F$ by $F^{2e}$ and assume that $F=F_0$.  
Let $y$ be a conjugate of $x$ in a maximal torus $\bT$ that is
$F$-invariant so that $F(t)=t^q$ for all $t \in \bT$.  
Thus, $F$ fixes $\langle y \rangle$ and so $F^{p-1}$ fixes $y$, whence
$C$ is $F$-stable. 
 \end{proof}
 
Note that the proof goes through verbatim if we only assume that $\bH$ is
reductive and that the derived group is simply connected with fewer than $p$
simple factors.

Returning to the proof of Theorem~\ref{thm:bsas} we see in particular,
if $p$ does not divide the order of the center of the simply
connected algebraic group $\bH$ in the same isogeny class with $\bG$, this
shows that $d$ is conjugate to an element of $X(q_0)$ (and since the
centralizer of $d$ covers $X(q)/S(q)$, this conjugation is via an element
of $S(q)$). Next we claim that some conjugate of
$c$ normalizes but does not centralize some conjugate of $X(q_0)$.
Since any two subgroups of field automorphisms of order  $p$ are conjugate via
an element of
$X(q)$, it follows that $c^{X(q)} \cap X(q_0)$ consists of more than one
conjugacy class. Therefore $c$ has more than one fixed point on $X(q)/X(q_0)$,
whence the result. Thus, choosing some subgroup $Y$ of $X(q)$ with
$Y \cong S(q_0)$, we may assume that each of $c$ and $d$ normalizes but does
not centralize $S(q_0)$, whence the result follows by induction.

So we will only need to consider field automorphisms in the case that
$S=\PSU_n(q)$ or $\PSL_n(q)$ with $p$ dividing $n$, and
these cases will be handled in the next subsection.

\subsubsection{Classical groups}
A) We first handle the case where $S=\PSL^\eps_p(q)$ with $p$ dividing $q-\eps 1$
and $c$ is an irreducible $p$-element. In particular, $c$ is semisimple
regular. First suppose that $d$ is semisimple.
By a minor variation of Gow's result \cite{Gow}, we see that $cd^g$
can be any regular semisimple element of $G$ in the coset $cdS$.
In particular, $cd^g$ need not be a $p$-element. By choosing
$cd^g$ to have order as large as possible in the torus acting irreducibly
on a hyperplane, we see that $\langle c, d^g \rangle$ need not be solvable
(for example, using the main result of \cite{GPPS}).

Suppose now that $d$ is a field automorphism, and 
let $T$ be a maximally split torus of $S$. Then $N_G(T)$ contains a Sylow
$p$-subgroup of $G$. Note that $N_G(T)/C_G(T) \cong \fS_p$ and both $c$
and~$d$ are conjugate to elements in $N_G(T) \setminus{C_G(T)}$ (this is
obvious for $c$, and for $d$ we can apply \cite[7.2]{GL}). Hence the result
follows by applying Lemma \ref{lem:alt} to $N_G(T)/C_G(T)$.

\medskip
B) Now let $S$ be any (simple) classical group with natural module $V$ of
dimension~$e$ defined over $\FF_{q_1}$. 
By our earlier results, it suffices to assume that $c$ is
semisimple and $d$ is either a field automorphism or a semisimple element.
Moreover, since $c^S$ is invariant under all diagonal automorphisms, by the
remark above, we can work with any conjugacy class of field automorphisms of
order $p$.

Let $m$ be the dimension of an irreducible module for an element of order $p$.
Then the case where $p=m$ and $m$ is the order of $q_1$ modulo $p$ has 
already been treated in A). 
Note that every semisimple element of order $p$ stabilizes
an $m$-dimensional subspace $W$ that is either nondegenerate or totally
singular (furthermore, the type is independent of the element).

Suppose that $W$ is totally singular. Then we may
assume that $c, d$ both normalize the stabilizer of $W$.
If $m=1$, then $c, d$ both normalize the stabilizer of a singular $1$-space
and the result follows by induction. So assume that $m > 1$. As $W$ is
totally singular, then by construction, we see that $c, d$ induce nontrivial
automorphisms on $\GL(W)$ (and since $p \ge 5$, $\SL(W)$ is quasisimple),
whence the result follows by induction.

Suppose that $W$ is nondegenerate. The same argument applies unless the
stabilizer of $W$ is not essentially simple. This only happens if $m =2$
and $S$ is an orthogonal group (and so we may assume that $e \ge 7$).
In this case, we see that $c, d$ will each stabilize a nondegenerate space
of the same type of either dimension $4$ or $6$ and we argue as above.

\subsubsection{Exceptional groups}
By the results above, we may assume that $c, d$ are both semi\-simple 
elements in $S$ of order $p$ (with $p \ge 5$). 
%In particular, it follows that a Sylow $p$-subgroup $P$
%of $\Inndiag(S)$ is contained in $S$. 
We may also assume that $P$ is not
cyclic (since that case is handled by \cite{GR} and \cite{guest}).
In particular, the result follows for
$S=\tw2B_2(q^2)$ or ${^2}G_2(q^2)$ since there $P$ is always cyclic.

If $S= \tw2F_4(2)'$, the result follows by a straightforward computation
(the only prime to consider is $p=5$).
Suppose that $S=\tw2F_4(q^2)$, $q^2>2$. It follows by \cite{MaF} that $P$
will either be contained in a subgroup $\tw2B_2(q^2) \wr 2$ or
$\Sp_4(q^2)$. In either case, we see that conjugates of $c, d$ will normalize
but not centralize a simple subgroup and the result follows by induction.

Suppose $S = G_2(q)$. Since $p \ge 5$ and $P$ is non-cyclic, we 
see that $p|(q^2-1)$ and $q \geq 4$. Now we can embed
$P$ in a subgroup $R \cong \SL_3(q)$ or $\SU_3(q)$ of $S$ and apply the 
previous results to $R$.  

Next suppose that $S=\tw3D_4(q)$. If $p|(q^2-1)$, then we can argue
as in the case of $G_2(q)$. The remaining cases are when $p$ divides 
$\Phi_3(q)$ or $\Phi_6(q)$.
One cannot find a good overgroup in these cases, but using \Chevie, we see that
$c^Gd^G$ hits any regular element in a torus of order
dividing $\Phi_{12}(q)$. In particular, $cd^g$ need not be a $p$-element.
By considering the maximal subgroups \cite{Kl}, it also
follows that $S=\langle c, d^g \rangle$ for some $g$.

\medskip
The standing hypothesis for the rest of this subsection is the following:

$\bullet$ $S$ is a simple exceptional Lie-type group, of type $F_4$, $E_6$,
$\tw2 E_6$, $E_7$, or $E_8$, over $\FF_q$;

$\bullet$ $c$ and $d$ are semisimple $p$-elements in $S$ and
the Sylow $p$-subgroups of $S$ are not cyclic.

Slightly changing the notation, we will view $S = S(q)$ as $(\bG^F)'$, where
$\bG$ is a simple algebraic group of adjoint type over $\overline{\FF}_q$
with a Steinberg endomorphism $F:\bG\to\bG$, and $W$ is the Weyl group of $\bG$.

The basic idea to prove Theorem \ref{thm:bsas} for $S$ is the following:

\begin{lem}   \label{lem:red}
 Assume Theorem~\ref{thm:bsas} holds for all non-abelian simple groups of
 order less than $|S|$. To prove Theorem \ref{thm:bsas} for semisimple 
 elements $c,d \in S$, it suffices to find a subgroup $D < S$ with
 the following properties:
 \begin{enumerate}[\rm (a)]
  \item $D = D_1 \circ \ldots \circ D_t$ is a central product of
   $t \leq 3$ quasisimple subgroups $D_i$ with $p$ coprime to $|Z(D)|$;
  \item each $S$-conjugacy class of elements of order $p$ intersects $D$; and
  \item either $N_S(D)$ acts transitively on $\{D_1, \ldots ,D_t\}$,
   or $t = 2$ and an $S$-conjugate of $D_1$ is contained in $D_2$.
 \end{enumerate}
\end{lem}

\begin{proof}
1)  By (b), we may assume that $c, d \in D$.
Suppose that there is some $j$ such that neither $c$ nor $d$
centralizes $D_j$. Then we can embed $c$ and $d$ in the almost simple
group $N_D(D_j)/C_D(D_j)$ with socle $D_j/Z(D_j)$. Since Theorem
\ref{thm:bsas} holds for $D_j/Z(D_j)$, we are done.

2)  Since $p\not{|}\,|Z(D)|$,
we are done if $t = 1$. Suppose $t = 2$. In view of 1) we may assume that
$c \in C_D(D_1) = Z(D)D_2$ (in particular, $c$ does not centralize $D_2$
and $c \in D_2$ since $p\not{|}\,|Z(D)|$).
Now if $d$ does not centralize $D_2$, we are also done. So we may assume
that $d \in C_D(D_2) = Z(D)D_1$, whence $d \in D_1$. By the assumptions,
there is some $s \in S$ such that $d^s \in D_2$. Now we can apply
Theorem~\ref{thm:bsas} to the images of $c$ and $d^s$ in $D_2/Z(D_2)$.

Finally suppose that $t = 3$. As above, we may assume that
$c \in E := D_2 \circ D_3$. If $d$ does not centralize $E$, then we can
embed both $c$ and $d$ in $N_D(E)/C_D(E)$ and repeat the $t = 2$ argument.
On the other hand, if $d \in C_D(E) = Z(D)D_1$, then $d \in D_1$ and
some $S$-conjugate $d^s$ lies in $D_2 < E$, and so $d^s$ does not centralize
$E$. Hence we are again done.
\end{proof}

The rest of this subsection is to produce a subgroup $D$ satisfying the
conditions set in Lemma \ref{lem:red}. In the following table we list
such a subgroup $D$. In all cases but the lines with $D = F_4(q)$,
$D$ is taken from \cite[Table 5.1]{LSS}, so that $N_{S}(D)$ is a
subgroup of maximal rank. In all cases, we choose $e$ smallest possible
such that $p|\Phi_e(q)$, and list the largest power $\Phi_e^l$ that divides
the order polynomial of $(\bG,F)$. According to \cite[Thm.~25.11]{MT},
$\bG^F$ has a unique conjugacy class of tori $T$ of order 
$\Phi_e^l(q)$. Moreover, by \cite[Lemma~4.5]{Za}, every $p$-element of 
$\bG^F$ of order at most the $p$-part of $\Phi_e(q)$ is conjugate to 
an element in $T$. In all cases, we choose $D$ so that it contains a 
$\bG^F$-conjugate of $T$ and $p$ is coprime to $|Z(D)|$.
%holds (for
%instance, if $p=5$ divides $q-1$ then $e = 1$, whence we do not consider
%$D = Z_5 \cdot (\PSL_5(q) \times \PSL_5(q))$ for $S = E_8(q)$). 
Also, all the Lie-type groups appearing in the third column are {\it simple}
non-abelian (here we are slightly abusing the notation, using
$E_6(q)$ and $\tw2 E_6(q)$ to denote their non-abelian composition factors).

\vskip1pc
\begin{center}
\begin{tabular}{|l|c|l|}\hline
$\bG^F$ & $\Phi_e^{l}$ & $D$ \\ \hline
$F_4(q)$ & $\Phi_1^4$, $\Phi_2^4$, or $\Phi_4^2$ & $Z_{(2,q-1)} \cdot
           \OO_9(q)$ \\
         & $\Phi_3^2$ or $\Phi_6^2$ & $\tw3 D_4(q)$ \\ \hline
$E_6(q)$ & $\Phi_1^6$ & $Z_{(2,q-1)} \cdot (\PSL_2(q) \times \PSL_6(q))$ \\
         & $\Phi_2^4$, $\Phi_4^2$, or $\Phi_6^2$ & $F_4(q)$ \\
         & $\Phi_3^3$ & $Z_{(3,q-1)} \cdot (\PSL_3(q) \times \PSL_3(q)
           \times \PSL_3(q))$ \\ \hline
$\tw2 E_6(q)$ & $\Phi_1^4$, $\Phi_3^2$, or $\Phi_4^2$ & $F_4(q)$ \\
              & $\Phi_2^6$ & $Z_{(2,q-1)} \cdot (\PSL_2(q) \times \PSU_6(q))$ \\
              & $\Phi_6^3$ & $Z_{(3,q+1)} \cdot (\PSU_3(q) \times \PSU_3(q)
                \times \PSU_3(q))$ \\ \hline
$E_7(q)$ & $\Phi_1^7$ or $\Phi_4^2$ & $Z_{(4,q-1)/(2,q-1)} \cdot \PSL_{8}(q)$ \\
         & $\Phi_2^7$ & $Z_{(4,q+1)/(2,q-1)} \cdot \PSU_{8}(q)$ \\
         & $\Phi_3^3$ & $Z_{(3,q-1)} \cdot E_{6}(q)$ \\
         & $\Phi_6^3$ & $Z_{(3,q+1)} \cdot \tw2 E_{6}(q)$ \\ \hline
$E_8(q)$ & $\Phi_1^8$, $\Phi_2^8$, $\Phi_4^4$, or $\Phi_8^2$ &
             $Z_{(2,q-1)} \cdot \OO^+_{16}(q)$ \\
         & $\Phi_3^4$ & $Z_{(3,q-1)} \cdot (\PSL_3(q) \times E_6(q))$ \\
         & $\Phi_5^2$ & $Z_{(5,q-1)} \cdot (\PSL_{5}(q) \times \PSL_5(q))$ \\
         & $\Phi_6^4$ & $Z_{(3,q+1)} \cdot (\PSU_3(q) \times \tw2 E_6(q))$ \\
         & $\Phi_{10}^2$ & $Z_{(5,q+1)} \cdot (\PSU_{5}(q) \times \PSU_5(q))$ \\
         & $\Phi_{12}^2$ & $\tw3 D_{4}(q^2)$ \\ \hline
\end{tabular}
\end{center}

\vskip1pc

To check the condition (3) of Lemma~\ref{lem:red}, we need to
work with the extended Dynkin diagram of $\bG$. Fix an orthonormal basis
$(e_1,\ldots,e_8)$ of the Euclidean space $\RR^8$ and let
$$\al_1 = (e_1+e_8-\sum^7_{i=2}e_i)/2,~~
  \al_2=e_2+e_1,~~\al_i=e_{i-1}-e_{i-2} ~~(3 \leq i \leq 8),~~\al'_8=-e_8-e_7,$$
so that $\al_1, \ldots ,\al_j$ are the simple roots of the root system of type
$E_j$, $6 \leq j \leq 8$, and $(\al_1, \ldots,\al_8,\al'_8)$ forms the extended
Dynkin diagram $E_8^{(1)}$ of type $E_8$. Also, let $\al'_6$ be chosen such
that $(\al_1, \ldots,\al_6,\al'_6)$ forms the extended Dynkin diagram
$E_6^{(1)}$ of type $E_6$.

Certainly, the condition (3) in Lemma~\ref{lem:red} needs to be verified only
when $D$ is not quasisimple. These cases are considered below, where we will
construct certain explicit automorphisms of the Dynkin diagram.

$\bullet$ $\bG^F = E_6(q)$. Let $\omega$ denote a graph automorphism of order
$3$ of $E_6^{(1)}$. Observe that it is induced by an element of $W$, whence by
some element $s \in S$. If $D$ is of type $A_1+A_5$, then $\omega$ sends
$\al'_6$ to $\al_1$ or $\al_5$, and so it sends the $A_1$-subgroup $D_1$ to
a subgroup of the $A_5$-subgroup $D_2$. If $D$ is of type $3A_2$, then $\omega$
permutes the three $A_2$-subgroups $D_i$ of $D$ cyclically.

$\bullet$ $\bG^F = \tw2 E_6(q)$. Let $\tau$ denote the unique graph
automorphism of order $2$ of the Dynkin diagram $E_6$ (which also acts on
$E_6^{(1)}$), so that $\bG^F$ is constructed using
$\tau$. If $D$ is of type $A_1 + \tw2 A_5$, then certainly the $A_1$-subgroup
$D_1$ (corresponding to
$\al'_6$) is $S$-conjugate to the $A_1$-subgroup labeled by
$\al_4$ of the $\tw2 A_5$-subgroup $D_2$. Assume now that
$D$ is of type $3(\tw2 A_2)$. Observe that $\tau$ is central in a Sylow
$2$-subgroup of the full automorphism group $Z_2 \times W$ of the root system of
type $E_6$. Hence it commutes with a $W$-conjugate of $\gamma$, the
automorphism that interchanges $\al_1$ with $\al_3$, $\al_5$
with $\al_6$, and $\al_2$ with $\al'_6$. So without loss we may assume
$D$ is constructed using this particular graph automorphism $\gamma$.
In this case, the order $3$ automorphism $\omega$ commutes with $\gamma$ and
permutes the three $\tw2 A_2$-subgroups $D_i$ of $D$ cyclically.

$\bullet$ $\bG^F = E_8(q)$. If $D$ is of type $A_2 + E_6$, then certainly the
$A_2$-subgroup $D_1$ (corresponding to $\al_8$ and $\al'_8$) is $S$-conjugate
to the $A_2$-subgroup labeled by $\al_1$ and $\al_3$ of the $E_6$-subgroup
$D_2$. Assume now that $D$ is of type $\tw2 A_2 + \tw2 E_6$. One can check
that $D$ can be constructed using the element
$$\beta~:~\al_1 \leftrightarrow \al_6,~\al_2 \mapsto \al_2,~
          \al_3 \leftrightarrow \al_5,~\al_4 \mapsto \al_4,~
          \al_7 \mapsto e_6+e_7,~\al_8 \leftrightarrow \al'_8$$
in $W$; in particular, $\beta$ fixes $\al'_6$. Applying the previous case to
$D_2 \cong \tw2 E_6(q)$, we see that $D_2$ contains a $\tw2 A_2$-subgroup
(labeled by $\al_5$ and $\al_6$). The latter is $S$-conjugate to $D_1$, the
$\tw2 A_2$-subgroup labeled by
$\al_8$ and $\al'_8$, via conjugation by the element
$$\delta~:~e_1 \mapsto e_1,~e_2 \mapsto e_2,~
           e_3 \leftrightarrow e_8,~e_4 \leftrightarrow -e_7,
           ~e_5 \leftrightarrow -e_6$$
in $W$, and so we are done.

Next, observe that the element
$$\begin{array}{lll}
  \varphi &:&\al_1 \mapsto \al_6 \mapsto \al_2 \mapsto \al'_8 \mapsto \al_1,
           ~\al_3 \mapsto \al_7 \mapsto \al_4 \mapsto \al_8 \mapsto \al_3,\\
           & & \al_5 \mapsto (-e_1-e_2-e_3+e_4+e_5+e_6-e_7+e_8)/2\end{array}$$
in $W$ interchanges the two $A_4$-components of $E_8^{(1)}$, and $\varphi^2$
induces the graph automorphism of each of these $A_4$-component. Since $F$ acts
trivially on $E^{(1)}_8$, we now see that $\varphi$ interchanges the two
$A_4$-subgroups $D_i$ if $D$ is of type $2A_4$, and $\varphi$ interchanges the
two $\tw2 A_4$-subgroups $D_i$ if $D$ is of type $2(\tw2 A_4)$.

We have therefore completed the proof of Theorem~\ref{thm:bsas}.

\subsection{Further variations on Baer--Suzuki}  \label{sec:bsgen}

We can now prove:

\begin{thm}   \label{thm:bsgen}
 Let $p \ge 5$ be prime. Let $G$ be a finite group. Let $C, D$ be conjugacy
 classes of $G$ with $G=\langle C \rangle = \langle D \rangle$.
 If $c^id^j$ is a $p$-element for all $(c,d) \in C \times D$ and all integers
 $i, j$, then $G$ is a cyclic $p$-group.
\end{thm}

\begin{proof}
First note that a $p$-group generated by a single conjugacy class is
cyclic (pass to the Frattini quotient to see that the Frattini quotient of $G$
and so $G$ are cyclic). Consider a minimal counterexample $(G, C, D)$.

We claim that $G$ has a unique minimal normal subgroup $N$. If $N_1$ and $N_2$
are distinct minimal normal subgroups, then by minimality, $G/N_i$ is a
$p$-group for each $i$, whence $G$ is a $p$-group, whence the claim.

By induction, $G/N$ is a cyclic $p$-group. If $N$ is a $p$-group, then so is
$G$ and the result follows. Assume that $N$ is an elementary abelian $r$-group
for some prime $r \neq p$, and let $P$ be a Sylow $p$-subgroup of $G$.
Choose $c, d \in P$, whence $d = c^i$ for some $i$.
Since $N$ is the unique minimal normal subgroup of $G$ and
$G/N = \langle dN \rangle$, $d$ acts irreducibly and nontrivially on $N$. It
follows that $d^N = dN$. In particular, we can find $x,y \in N$ such that
$d^x = dy \neq d$. Now $c^{-i}d^x = y$ is not a $p$-element, a contradiction.

So $N$ is a direct product of copies of a non-abelian simple group $L$.
%(that has order divisible by $p$).
Replacing $C$ and $D$ by $C^q$ and $D^q$,
we may assume that $|G/N| = p$. If $N$ is simple, then
$G$ is almost simple and Theorem~\ref{thm:bsas}(b) applies. So we may
assume that $N$ is a direct product $L_1 \times \cdots \times L_p$ and
that an element of $C$ or $D$ conjugates $L_i$ to $L_{i+1}$ for
$1 \le i < p$. Replacing the elements of $D$ by a power prime to $p$,
we may assume that $CD \subset N$. Choose $(c,d) \in C \times D$. Write
$c = (x_1, \ldots, x_p)\rho$ where $x_i \in \Aut(L_i)$ and $\rho \in \Aut(N)$
permuting the $L_i$ in a cycle. So $d=\rho^{-1}(y_1, \ldots, y_p)$ with
$y_i \in \Aut(L_i)$. Choosing $h = (z,1, \ldots ,1) \in N$ with $z$ running
over $L$, we see that $cd^h \in N$, whose first
coordinate is equal to $x_1y_1z$ and so it also runs over $L$.
In particular $cd^h$ need not be a $p$-element for all $h \in N$.
\end{proof}

We now want to weaken the hypothesis that $G = \langle C \rangle
=\langle D \rangle$ in Theorem~\ref{thm:bsgen}. To do so,
we have to weaken slightly the conclusion.

We first need the following result:

\begin{lem}   \label{lem:primetop}
 Let $p \ge3$ be prime. Let $G$ be a finite group with a Sylow $p$-subgroup
 $P$ and a normal $p$-complement $N$. Assume that
 $P = \langle C \rangle = \langle D \rangle$ for $C, D \subset P$, and that
 $\langle c^x, d \rangle$ is a $p$-group for all $x \in N$ and 
 $(c,d) \in C \times D$. Then $G = N \times P$.
\end{lem}

\begin{proof}
Observe that the hypotheses imply that $N = C_N(c)C_N(d)$ for all 
$(c,d) \in C \times D$. Indeed, for all $x \in N$ we have that 
$\langle c^x, d \rangle$ is a $p$-group. Thus, there exists $y \in N$ with
$c^{xy}, d^y  \in P$. Since $N \cap P = 1$, it follows that $y \in C_N(d)$ and 
$xy \in C_N(c)$, whence $x \in C_N(c)C_N(d)$.

By way of contradiction, assume that $[P,N] \ne 1$. Note that if $R$ is 
a Sylow $r$-subgroup of $N$,
then $G= N_G(R)N$, whence $N_G(R)$ contains a Sylow $p$-subgroup of $G$.
Thus, $P$ normalizes a Sylow $r$-subgroup $R$ of $G$ for each prime
divisor $r$ of $|N|$.  So for some $r$, $P$ does not centralize $R$.
Thus, without loss we may assume that $N$ is an $r$-group for some prime
$r$.  By passing to a quotient,  we may first assume that
$N$ is elementary abelian and then that $P$ acts
irreducibly and nontrivially on $N$.

Now view $N$ as an absolutely irreducible $\FF P$-module where
$\FF :=\End_P(N)$. We can extend scalars and work over an algebraically
closed field. So $N = \Ind_M^P(W)$ for some irreducible $M$-module $W$
with $M$ a maximal subgroup 
of $P$. Since $P/M$ is cyclic and $N$ is irreducible over $P$, we have that 
$N= W_1 \oplus \ldots \oplus W_p$ where the $W_i$ are pairwise non-isomorphic 
irreducible $M$-modules. Choosing $c \in C \setminus{M}$, we see that 
$c$ permutes the $W_i$ transitively, whence $\dim C_N(c) \le (\dim N)/p$.

Thus we have found $c \in C$ with 
$\dim_{\FF_r} C_N(c) \le (\dim_{\FF_r}N)/p$
and similarly for some $d \in D$. For this choice of $(c,d)$,
$|C_N(c)C_N(d)| \le |N|^{2/p} < |N|$, a contradiction.
\end{proof}

We can now prove another variation on Baer--Suzuki, which is 
Theorem \ref{thm:bsthm} in the introduction. Note that this includes
the usual Baer--Suzuki theorem (for $p \ge 5$) by taking $C=D$.

\begin{thm}   \label{thm:bsplus}
 Let $G$ be a finite group and $p \ge 5$ prime. Let $C$ and $D$ be
 normal subsets of $G$ such that $H:=\langle C \rangle = \langle D \rangle$.
 If $\langle c, d \rangle$ is a $p$-group for all $(c,d) \in C \times D$,
 then $H \le O_p(G)$.
\end{thm}

\begin{proof}
1) Let $G$ be a counterexample of minimal order. By minimality, $G=H$.
By the usual argument, we see that $G$ must have a unique minimal normal
subgroup $N$. Let $P$ be a Sylow $p$-subgroup of $G$.

By induction, $G/N$ is a $p$-group, whence $G = NP$ and $N$ is not a $p$-group.
For any $c \in C$, since $\langle c \rangle$ is a $p$-subgroup, 
we can find $x \in N$ such that $c' := c^x \in C \cap P$. It follows that
$Nc' = Ncx = N(cxc^{-1})c = Nc$, and so $c \in N \langle C \cap P \rangle$. 
Thus $G = N \langle C \cap P \rangle$, and similarly, 
$G = N \langle D \cap P \rangle$.   

\medskip
2) Suppose that $N$ is a $p'$-group. Then $N \cap P = 1$ and 
$NP = N \langle C \cap P \rangle$ by 1), whence 
$P = \langle C \cap P \rangle$
and similarly, $P = \langle D \cap P \rangle$.
Applying Lemma \ref{lem:primetop}, we see that $P \lhd G$, a contradiction.

Thus we may assume that $N=L_1 \times \cdots \times L_t$
where $L_i \cong L$, a non-abelian simple group (of order divisible by $p$).
Let $Q := P \cap N = Q_1 \times \cdots \times Q_t$ with
$Q_i \leq L_i$, and let $T:=N_G(Q)=XP$, where $X=X_1 \times \cdots \times X_t$
with $X_i:=N_{L_i}(Q_i)$.
By a result of Glauberman--Thompson \cite[Thm.~X.8.13]{HB} (see also
\cite{gmn}), it follows that $X_i \ne Q_i$.

Now consider $T/Q = (X/Q)(P/Q)$.  Then $P/Q \cong G/N$ is generated
by the images of $C \cap P$ and also by the images of $D \cap P$ by 1),
and $X/Q$ is a $p'$-group.  So by Lemma~\ref{lem:primetop} applied to
$T/Q$, $P/Q$ must
centralize $X/Q \cong (X_1/Q_1) \times  \cdots \times (X_t/Q_t)$.  
But $X_i \neq Q_i$ and $P$ permutes the $L_i$, hence  
$P$ must normalize each $L_i$. Since $N$ normalizes
each $L_i$, this implies that $L_i$ is normal in $G = NP$.
Recall that $N$ is the unique minimal normal subgroup of $H$. Thus, we have 
shown that $N=L_1$ is simple and so $G$ is almost simple.  Now we
have a contradiction by Theorem \ref{thm:bsas}(a).
\end{proof}

There is a version of the previous result for linear groups.

\begin{cor} \label{cor:linearbs}
Let $k$ be a field of characteristic $p$ with $p=0$ or
$p > 3$.    Let $G$ be a subgroup of $\GL_n(k)$.  If $C$ and $D$
are normal unipotent subsets of $G$ with $H:=\langle C \rangle =
\langle D \rangle$
such that $\langle c, d \rangle$ is unipotent for all $(c,d) \in C \times D$,
then $H$ is a normal unipotent subgroup of $G$.
\end{cor}

\begin{proof}
There is no harm in assuming that $k$ is algebraically closed and that $G=H$.
Since the condition that $\langle c, d \rangle$ is unipotent is a closed
condition, it suffices to prove the result in the case where $G$, $C$ and
$D$ are replaced by their Zariski closures.  
So $G=\bG$ is an algebraic group.  We may furthermore assume that the
unipotent radical of $\bG$ is trivial.
In particular, the connected component $\bG^\circ$ of $\bG$ is reductive.
By the result for finite groups, $\bG/\bG^\circ$ is a $p$-group (in particular
if $p=0$, $\bG$ is connected). If $\bG^\circ$ is trivial, the result follows.
Let $\bB$ be a Borel subgroup of $\bG^\circ$ with unipotent radical $\bU$.
Then $N_\bG(\bB)$ covers $\bG/\bG^\circ$.
Let $\bP$ be a maximal (necessarily closed) unipotent subgroup of 
$N_\bG(\bB)$ (so $\bU \le \bP$), and let $\bT$ be a maximal torus of $\bB$.
Then $N_\bG(\bB)/\bU = \bT.(\bP/\bU)$.  Note
that $\bP/\bU$ is generated by $C\bU/\bU$ (as in our earlier arguments).
For any $m\ge1$ let $\bT[m]$ be the $m$-torsion subgroup of $\bT$. Note that 
$\bT[m]$ is a finite group. Applying Lemma~\ref{lem:primetop}, it follows that
$[\bP,\bT[m]] \le \bU$. Since $\bT$ is the closure of its torsion
subgroup, $[\bP,\bT] \le \bU$.  Thus, 
$\bG$ normalizes each simple component of $\bG^\circ$ and so we are reduced
to the almost simple case. However a simple algebraic group in
characteristic $p \neq 2,3$ has no outer automorphisms of order $p$ and
so $\bG$ is simple. Now the result follows  by
Theorem \ref{thm:algebraic}.
\end{proof}

\vskip 2pc
%%%%%%%%%%%%%%%%%%%%%%%%%%%%%%%%%%%%%%%%%%%%%%%%%%%%%%%%%%%%%%%%%%%%%%%%%

\end{document}